\documentclass[12pt]{article}

\usepackage[acronym]{glossaries}  
\usepackage[top=25.4truemm,bottom=25.4truemm,left=25.4truemm,right=25.4truemm]{geometry}
\usepackage{booktabs}
\usepackage{mathptmx}      %
\usepackage{amsmath}
\usepackage{amssymb}
\usepackage{amsthm}
\usepackage{amscd}
\usepackage{enumitem}
\usepackage{tikz-cd}
\usepackage{bm}
\usepackage{hyperref}
\usepackage{algorithm}
\usepackage{algpseudocode}
\usepackage{authblk}

\usepackage{xcolor}
\hypersetup{%
setpagesize=false,
 bookmarksnumbered=true,%
 bookmarksopen=true,%
 colorlinks=true,%
 linkcolor=black,
 citecolor=black,
}
\usepackage{graphicx, caption}
\usepackage{mathrsfs}

\theoremstyle{definition}
\newtheorem{theorem}{Theorem}[section]

\newtheorem{proposition}[theorem]{Proposition}
\newtheorem{corollary}[theorem]{Corollary}
\newtheorem{lemma}[theorem]{Lemma}

\newtheorem{remark}[theorem]{Remark}
\newtheorem{definition}[theorem]{Definition}

\def\note[#1]#2{
\vspace{2mm}\ \\
{\noindent\bf \footnotesize {\underline{Memo} (#1)}\hrulefill}\\{\footnotesize #2}\\
\hrulefill\\\vspace{2mm}
}

\def\todo#1
{
  {\par\noindent \textcolor{red}{\bf ToDo: }}#1
  {\textcolor{red}{\rule[-1pt]{2pt}{9pt}}\vspace{5pt} \par\noindent}
}

\numberwithin{equation}{section}

\providecommand{\keywords}[1]
{
  \noindent{\small	
  \textbf{\textit{Keywords---}} #1}
}

\providecommand{\MSC}[1]
{
  {\small
  \noindent\textbf{\textit{Mathematics Subject Classification (2020)---}} #1\par}
}

\newcommand{\NormCpt}[2]{|#1|_{#2}}

\title{Finite-dimensional approximations of push-forwards on locally analytic functionals}
\date{}
\author[1,2]{Isao Ishikawa\thanks{ishikawa.isao.5s@kyoto-u.ac.jp}}

\affil[1]{Kyoto University}
\affil[2]{RIKEN}

\begin{document}

\maketitle

\begin{abstract}
    This paper develops a functional-analytic framework for approximating the push-forward induced by an analytic map from finitely many samples. 
    Instead of working directly with the map, we study the push-forward on the space of locally analytic functionals and identify it, via the Fourier--Borel transform, with an operator on the space of entire functions of exponential type. 
    This yields finite-dimensional approximations of the push-forward together with explicit error bounds expressed in terms of the smallest eigenvalues of certain Hankel moment matrices. 
    Moreover, we obtain sample complexity bounds for the approximation from i.i.d.~sampled data.
    As a consequence, we show that linear algebraic operations on the finite-dimensional approximations can be used to reconstruct analytic vector fields from discrete trajectory data. 
    In particular, we prove convergence of a data-driven method for recovering the vector field of an ordinary differential equation from finite-time flow map data under fairly general conditions.
    
\end{abstract}

\medskip

\MSC{
    Primary 37C30; Secondary 32E30, 30H20, 41A25, 46F15.
}

\section{Introduction}\label{sec: intro}

Data-driven methods for the analysis and control of dynamical systems have attracted considerable attention in recent years.
A prominent role is played by operator-theoretic approaches based on the Koopman operator, and in particular by the extended dynamic mode decomposition (EDMD) and its variants; see, for instance, \cite{williamsKernelbasedMethodDatadriven2015,bruntonModernKoopmanTheory2022} and the references therein.
Recent rigorous convergence results for EDMD in analytic settings include \cite{akindjiConvergencePropertiesDynamic2026}, where spectral convergence is established for analytic full branch interval maps under a suitable relation between the number of observables and the number of collocation points.
These methods approximate the Koopman operator from finite discrete trajectory data, and have been successfully applied to a wide range of problems in fluid mechanics, molecular dynamics, control, and beyond.

From the viewpoint of differential equations, however, it is often more natural and more informative to reconstruct the underlying map or vector field itself from the observed trajectories.
Given finitely many samples of an analytic map or of the flow map of an ordinary differential equation (ODE), one would like to recover the map or the vector field with quantitative error bounds, and to understand how the accuracy depends on the system and on properties of the sampling set.
The main purpose of this paper is to develop a functional-analytic framework that allows us to address these questions in a unified way for analytic dynamical systems.

Our basic point of view is to work with push-forwards on spaces of locally analytic functionals instead of working directly with the Koopman operator on a space of observables.
Throughout the paper we work with complex vector spaces $E = \mathbb{C}^d$ and $F = \mathbb{C}^r$, equipped with the standard Hermitian inner products $\langle \cdot, \cdot \rangle_E$ and $\langle \cdot, \cdot \rangle_F$, respectively.
We fix a compact neighbourhood $K_0 \subset E$ of $0_E$ which is \emph{absolutely convex}, that is, $ax+by \in K_0$ whenever $x,y \in K_0$ and $a,b \in \mathbb{C}$ with $|a|+|b|\le 1$.
The associated Minkowski functional
\[
  \NormCpt{z}{K_0} := \inf\{ r>0 : z/r \in K_0\}, \qquad z \in E,
\]
defines a norm on $E$ such that $\NormCpt{z}{K_0} = \|z\|_E$ when $K_0$ is the closed unit ball, and $z \in K_0$ if and only if $\NormCpt{z}{K_0}\le 1$.
We fix $p \in \mathbb{R}^d$ and denote $K := p + K_0 \subset E$.
Let $f\colon K \cap \mathbb{R}^d \to F$ be a map which admits a holomorphic extension (denoted again by $f$) to an open neighbourhood of $K$.
For an open or compact set $S \subset E$ we denote by $\mathcal{O}_E(S)$ the space of complex-valued holomorphic functions on $S$ endowed with the usual locally convex topology (see \cite[Section~1.4]{morimotoIntroductionSatosHyperfunctions1993}), and we use similar notation for $F$.
The analytic map $f\colon K \to F$ induces a continuous pull-back
\[
  f^* \colon \mathcal{O}_F(F) \longrightarrow \mathcal{O}_E(K), \qquad h \longmapsto h\circ f,
\]
and its dual map
\[
  f_* := (f^*)' \colon \mathcal{O}_E(K)' \longrightarrow \mathcal{O}_F(F)'
\]
is referred to as the \emph{push-forward} associated with $f$, where the topology of the dual spaces is the strong topology.
We also recall a slight modification of the Fourier--Borel transform $\overline{\mathbf{FB}}_E$ and its relation to entire functions of exponential type.
Writing
\[
  \mathbf{e}_\zeta(z) := e^{\zeta^* z}, \qquad \zeta,z \in E,
\]
the map $\overline{\mathbf{FB}}_E$ sends a locally analytic functional $\mu \in \mathcal{O}_E(K)'$ to the entire function
\begin{align*}
    \overline{\mathbf{FB}}_E(\mu)(\zeta) := \overline{\mu(\mathbf{e}_\zeta)}.
\end{align*}
By the Ehrenpreis--Martineau theorem, this yields a topological isomorphism between $\mathcal{O}_E(K)'$ and a space of entire functions of exponential type with respect to $K$, which we denote by ${\rm Exp}(E;K)$.
In addition, we work with the Fock space $H_E \subset \mathcal{O}_E(E)$, a reproducing kernel Hilbert space with kernel $k(z,w)=e^{\langle z,w\rangle_E}$, and exploit the continuous embedding ${\rm Exp}(E;K) \hookrightarrow H_E$ to represent the push-forward as an (unbounded) linear operator on the Hilbert space.

Within this framework, we construct finite-dimensional approximations of the push-forward from finitely many samples from $f$ and derive quantitative convergence rates.
The construction is based on finite-dimensional subspaces generated by jets at the base point $p$.
Roughly speaking, we show that the push-forward on locally analytic functionals can be approximated by a linear map on finite-dimensional subspaces of the Fock space whose matrix coefficients are computed from empirical data, and that the approximation error can be bounded explicitly in terms of the properties of $f$ and the distribution of the empirical data.

\medskip

Before stating our main results, let us introduce several notations.
We fix a numbering $ \{\alpha^{(i)}\}_{i=1}^\infty = \mathbb{Z}_{\ge 0}^d$ of  $\mathbb{Z}_{\ge 0}^d$ such that $\alpha^{(1)}$ is the zero vector, that $\alpha^{(i+1)}$ is the $i$-th elementary vector for $i=1,\dots, d$, and that $|\alpha^{(i)}| \le |\alpha^{(j)}|$ if $i \le j$.
For $i \in \mathbb{Z}_{>0} $, we define
    \begin{align*}
        u_{p,i}(z) &:= \frac{e^{-\|p\|_E^2/2}}{\sqrt{\alpha^{(i)}!}}(z - p)^{\alpha^{(i)}} e^{p^*z}.
    \end{align*}
Here, we use the multi-index notation: $z^\alpha := \prod_{i=1}^d z_i^{\alpha_i}$ and $\alpha! := \prod_{i=1}^d \alpha_i!$ for $\alpha = (\alpha_1,\dots,\alpha_d) \in \mathbb{Z}_{\ge 0}^d$ and $z=(z_1,\dots, z_d) \in \mathbb{C}^d$.
Also,  we equip $\mathbb{Z}_{\ge 0}^r$ with the numbering $\{\beta^{(i)}\}_{i=1}^\infty = \mathbb{Z}_{\ge 0}^{r}$ in the same manner, and for $q \in F$, we define
    \begin{align*}
        v_{q,i}(w) &:= \frac{e^{-\|q\|_F^2/2}}{\sqrt{\beta^{(i)}!}}(w - q)^{\beta^{(i)}} e^{q^*w}.
    \end{align*}
For $n \in \mathbb{Z}_{\ge 0}$, we set
\[
    r_n^E := \#\{\alpha \in \mathbb{Z}_{\ge 0}^d : |\alpha|\le n\}
    = \binom{n+d}{d},
    \qquad
    r_n^F := \#\{\beta \in \mathbb{Z}_{\ge 0}^r : |\beta|\le n\}
    = \binom{n+r}{r}.
\]
For $n \in \mathbb{Z}_{\ge 0}$, $z \in E$, and $w \in F$, we define (horizontal) vectors of $\mathbb{C}^{1 \times r_n^{E}}$ and $\mathbb{C}^{1 \times r_n^{F}}$ by
\begin{align*}
    \mathbf{u}_{p, n}(z) &:= \left(u_{p,1}(z), u_{p,2}(z), \dots, u_{p,r_n^{E}}(z)\right), \\
    \mathbf{v}_{q,n}(w) &:= \left(v_{q,1}(w), v_{q,2}(w), \dots, v_{q,r_n^{F}}(w)\right).
\end{align*}
For $Z = (z^1,\dots, z^N) \in E^N$ and $W = (w^1,\dots, w^N) \in F^N$, we define matrices $\mathbf{U}_{p,n,Z} \in \mathbb{C}^{N \times r_n^{E}}$  and $\mathbf{V}_{q,n,W} \in \mathbb{C}^{N \times r_n^{F}}$ by
\begin{align*}
\mathbf{U}_{p,n, Z}  := 
\begin{pmatrix}
    \mathbf{u}_{p,n}(z^1)\\
    \vdots\\
    \mathbf{u}_{p,n}(z^N)
\end{pmatrix},~~~ 
\mathbf{V}_{q,n, W} :=
\begin{pmatrix}
    \mathbf{v}_{q,n}(w^1)\\
    \vdots\\
    \mathbf{v}_{q,n}(w^N)
\end{pmatrix}.
\end{align*}

Suppose that $z^i \in K$ and $w^i = f(z^i)$ for $i=1,\dots, N$.
Then, for integers $m$ and $n$ with $m \le n$, we define $\widehat{\mathbf{C}}_{m,n,Z} \in \mathbb{C}^{r_m^{F} \times r_m^{E}}$ by the leftmost submatrix of $\mathbf{V}_{f(p),m,W}^* (\mathbf{U}_{p,n,Z}^*)^\dagger \in \mathbb{C}^{r_m^{F} \times r_n^{E}}$ of size $r_m^{F} \times r_m^{E}$:
\begin{align}
\mathbf{V}_{f(p),m,W}^* (\mathbf{U}_{p,n,Z}^*)^\dagger
=
\begin{pmatrix}
\widehat{\mathbf{C}}_{m,n,Z}  & * 
\end{pmatrix}, \label{regression matrix}
\end{align}
where $(\cdot)^\dagger$ indicates the Moore-Penrose pseudo-inverse.
For an integer $n \ge 0$ and a Borel measure $\mu$ on $E$, we define a matrix
\begin{align}
    \mathbf{D}_n(\mu)
    := \left( \int  \overline{z^{\alpha^{(i)}}} z^{\alpha^{(j)}} {\rm d}\mu(z)\right)_{i,j = 1,\dots, r^E_n}
    \label{moment matrix}
\end{align}
Let $\Lambda_n(\mu)$ be the smallest eigenvalue of $ \mathbf{D}_n(\mu)$.

\medskip

Our first theorem concerns the approximation of the push-forward from deterministic samples and identifies the convergence rate in terms of the smallest eigenvalues $\Lambda_n(\mu)$ of the Hermitian moment matrices \eqref{moment matrix}, which reduce to Hankel moment matrices when $\mu$ is supported on $\mathbb{R}^d$.
\begin{theorem}[A special version of Theorem \ref{thm: main thm}] \label{thm: main thm, intro}
    Let $\mu$ be a Borel probability measure on  $K_0 \cap \mathbb{R}^d$ absolutely continuous with respect to the Lebesgue measure on $K_0 \cap \mathbb{R}^d$.
    Let $\{x^i\}_{i=1}^\infty \subset {\rm supp}(\mu)$ and define $X_N := (p + x^1,\dots, p + x^N) \in K^N$.
    Let $\mathbf{C}_m$ be the representation matrix of $f_*|_{\mathfrak{D}_{p,m}}: \mathfrak{D}_{p,m} \to  \mathfrak{D}_{f(p),m}$ with respect to the basis $\big(\overline{\mathbf{FB}}_E^{-1}(u_{p,i}) \big)_{i=1}^{r_m^E}$ and $\big( \overline{\mathbf{FB}}_F^{-1}(v_{f(p),i}) \big)_{i=1}^{r_m^F}$.
    Assume that $N^{-1}\sum_{i=1}^N \delta_{x^i}$ weakly converges to $\mu$ as $N \to \infty$.
    Then, for any $m, n \in \mathbb{Z}_{\ge 0}$ with $m \le n$, the following inequality holds: 
    \begin{align}
          \limsup_{N \to \infty}\left\| \mathbf{C}_m - \widehat{\mathbf{C}}_{m,n,X_N}\right\|_{\rm Fr} 
          \lesssim 
          \sqrt{m!} \frac{R_\mu^n}{\Lambda_n(\mu)^{1/2}},
          \label{conv rate of pf op, intro}
    \end{align}
    where $R_\mu := \sup_{x \in {\rm supp}(\mu)}\NormCpt{x}{K_0}$.
\end{theorem}
We note that the linear subspace of $\mathcal{O}_E(K)'$ generated by the basis $\big(\overline{\mathbf{FB}}_E^{-1}(u_{p,i}) \big)_{i=1}^{r_m^E}$ coincides with the dual of the space of $m$-jets at $p$ (see Remark \ref{rem: jets}), where the jet is a canonically defined object on a manifold (see, for example, \cite[Section IV]{kolarNaturalOperationsDifferential1993}).

We also obtain quantitative sample complexity bounds for i.i.d.~sampled data from a certain distribution (see Corollary \ref{cor: main thm, iid rand. var.}).

\begin{remark}[Decreasing rate of $\Lambda_n(\mu)$]
    In a practical setting (for example, the density function is continuous on an open set), the growth rate of the inverse of $\Lambda_n$ is bounded above by an exponential function on $n$.
    In fact, for $p' \in \mathbb{R}^d$ and $\mathbf{r} := (r_1,\dots,r_d) \in \mathbb{R}_{>0}^d$, let $\Delta(p'; \mathbf{r}) := \{ x \in \mathbb{R}^d : |x_i - p_i'| \le r_i ~(i=1,\dots,d)\}$ be a rectangle in $\mathbb{R}^d$.
    Let $\rho := {\rm d}\mu / {\rm d}x$ be the Radon--Nikodym derivative of $\mu$ with respect to the Lebesgue measure on $\mathbb{R}^d$.
    Assume that $\mathop{\rm ess.inf}_{\Delta(p'; \mathbf{r})}\rho > 0$, which is true, for example, if $\rho$ is positive and continuous on an open subset.
    Then, $\Lambda_n(\mu_{\Delta(p'; \mathbf{r})}) \lesssim \Lambda_n(\mu)$ holds.
    Moreover, there exists $\eta > 0$ such that $\Lambda_n(\mu_{\Delta(p'; \mathbf{r})}) > \eta^{2n}$ for all $n$ (see Theorem \ref{thm: asymptotic formula for the minimal eigenvalues for general mu}).
    Thus, we see that 
    \begin{align*}
        \frac{R_\mu^n}{\Lambda_n(\mu)^{1/2}} \lesssim  \left( \frac{R_\mu}{\eta} \right)^n,
    \end{align*}
    and the right hand side of \eqref{conv rate of pf op, intro} converges to $0$ as $n \to \infty$ if $K_0$ is sufficiently large ($f$ is holomorphic on a sufficiently wide region). 
    For more details on the estimation of $\Lambda_n(\mu_{\Delta(p'; \mathbf{r})})$, see Appendix \ref{appendix: Asymptotics of the lowest eigenvalues of Hankel matrices}.
\end{remark}

\medskip

Our second main result is a reconstruction theorem for analytic maps, stated in terms of a data-driven analytic map $\widehat{f}_{m,n,Z_N}$ obtained from the finite-dimensional approximation of the push-forward.
It also contains, as a particular case, a convergence result for ``truncated'' least-squares polynomial approximations in the random sampling setting (see Theorem \ref{thm: convergence of least-squares polynomial}).

\begin{theorem}[A special version of Theorem \ref{thm: regression}]
\label{thm: main thm 2, intro}                                                                             
    Let $\mu$ be a Borel probability measure on  $K_0 \cap \mathbb{R}^d$ absolutely continuous with respect to the Lebesgue measure on $K_0 \cap \mathbb{R}^d$.
    Let $\{x^i\}_{i=1}^\infty \subset {\rm supp}(\mu)$ and define $X_N := (p + x^1,\dots,  p + x^N) \in K^N$.
    Assume that $N^{-1}\sum_{i=1}^N \delta_{x^i}$ weakly converges to $\mu$ as $N \to \infty$.
    For $i=1,\dots, r$ and $z \in E$, we define 
    \[
    \widehat{f}_{m,n,X_N,i}(z) := \mathbf{u}_{p,m}(z) \widehat{\mathbf{C}}_{m,n,X_N}^* \left(\frac{\partial \mathbf{v}_{f(p),m}}{\partial w_i}(0)\right)^*.\] 
    Let $\widehat{f}_{m,n,X_N} := (\widehat{f}_{m,n,X_N,1}, \dots, \widehat{f}_{m,n,X_N,r}): E \to F$.
    Then, for any $m, n \in \mathbb{Z}_{\ge 0}$ with $m \le n$ and compact subset $K_1 \Subset K_0$, we have
    \begin{align}
        &\limsup_{N \to \infty}\sup_{z \in p + K_1} \big\|\widehat{f}_{m,n,X_N}(z) - f(z) \big\|_{F}
        \le
        C\left(
        \sqrt{m!}
        \frac{R_\mu^n}{\Lambda_n(\mu)^{1/2}}
        + R_{K_1}^m
        \right)
    \end{align}
    where $R_{K_1} := \sup_{z \in K_1}\NormCpt{z}{K_0}$, $R_\mu := \sup_{x \in {\rm supp}(\mu)}\NormCpt{x}{K_0}$, and $C$ is a constant only depending on $f$, $p$, $K_0$, $K_1$, and $\mu$.
    
    Furthermore, let $L_\mu :=  \max\big(1, \sup_{x \in {\rm supp}(\mu)} \max_{i=1,\dots,d}|x_i|\big)$ and let $\delta \in (0,1)$.
    Assume that $x^1, x^2, \dots$ are i.i.d. random variables with the probability measure $\mu$ and that
    \begin{align*}
        N &\ge \frac{L_\mu^{4n} \cdot (r_n^E)^2 }{\Lambda_n(\mu)^2}\cdot 4\log(2/\delta)
    \end{align*}
    holds.
    Then,  with probability at least $1-\delta$, we have
    \begin{align}
    \sup_{z \in p+K_1} \big\|\widehat{f}_{m,n,X_N}(z) - f(z) \big\|_{F}    \le C\left(
    \sqrt{m!}
    \frac{R_\mu^n}{\Lambda_n(\mu)^{1/2}}
    + R_{K_1}^m
    \right).
    \end{align}
\end{theorem}

\medskip

The last main result is concerned with the reconstruction of analytic vector fields for ordinary differential equations 
\begin{align}
        \dot{z} &= V(z), \label{ODE}
\end{align}
from finite-time flow map data, where $V: K \to E$.
Here, we define the flow map $\phi^t: K \to E$ by $\phi^t(w) := z(t)$ where $z(t)$ is the solution of \eqref{ODE} with initial point $z(0)=w$.

Assuming that $p$ is an equilibrium of an analytic vector field $V$ and that the eigenvalues of the Jacobian matrix ${\rm d}V_p$ are real, we construct a holomorphic vector field $\widetilde{V}_{m,n,Z_N}$ from discrete samples of the time-$T$ map $\phi^T$ around $p$ and prove a quantitative convergence estimate.
\begin{theorem}[A special version of Theorem \ref{thm: reconstruction of vector field}]
\label{thm: main thm 3, intro}
Let $\mu$ be a Borel probability measure on  $K_0 \cap \mathbb{R}^d$ absolutely continuous with respect to the Lebesgue measure on $K_0 \cap \mathbb{R}^d$.
Let $\{x^i\}_{i=1}^\infty \subset {\rm supp}(\mu)$ and define $X_N := (p + x^1,\dots,  p + x^N) \in K^N$.
Assume that 
\begin{enumerate}[label=(\roman*)]
    \item $V(p) = 0_E$,
    \item the eigenvalues of the Jacobian matrix of $V$ at $p$ are real, \label{real eigenvalues}
    \item there exists $T>0$ such that for any initial point $w \in K$, the ODE \eqref{ODE} has a solution on $[0,T]$,    
    \item the measure $ N^{-1}\sum_{i=1}^N \delta_{x^i}$ weakly converges to $\mu$ as $N \to \infty$,
    \item there exists $\rho \in (0,1)$ such that $\rho > R_\mu/\liminf_{n \to \infty} \Lambda_n(\mu)^{1/(2n)}$ where  $R_\mu := \sup_{x \in {\rm supp}(\mu)}\NormCpt{x}{K_0}$.
\end{enumerate}
Let $f:= \phi^T$ and construct the matrix $\widehat{\mathbf{C}}_{m,n,X_N}$ for $f$ as in \eqref{regression matrix}.
For all sufficiently large $N$ for which $\widehat{\mathbf{C}}_{m,n,X_N}$ has no non-positive real eigenvalues, we define
\begin{align}
    \widehat{\mathbf{A}}_{m,n,X_N} := \frac{1}{T}\log \widehat{\mathbf{C}}_{m,n,X_N}.
\end{align}
where $\log$ is the principal matrix logarithm (see \eqref{matrix log} for the definition).
For $i=1,\dots, d$ and $z \in E$, we define 
\[
\widehat{V}_{m,n,X_N,i}(z) := \mathbf{u}_{p,m}(z) \widehat{\mathbf{A}}_{m,n,X_N}^* \left(\frac{\partial \mathbf{u}_{p,m}}{\partial z_i}(0)\right)^*.
\]
Let $\widehat{V}_{m,n,X_N} := (\widehat{V}_{m,n,X_N,1}, \dots, \widehat{V}_{m,n,X_N,d}): E \to E$.
Then, there exists a positive sequence $\{C_m\}_{m=0}^\infty \subset \mathbb{R}_{>0}$ such that for any compact set $K_1 \Subset K_0$, $m \in \mathbb{Z}_{\ge 0}$, and sufficiently large $n$, we have
\begin{align}
    & \limsup_{N \to \infty}\sup_{z \in p+K_1} \big\|\widehat{V}_{m,n,X_N}(z) - V(z) \big\|_E \nonumber \\
    & \le
    C' \left(
    C_m \rho^n
    + \frac{R_{K_1}^m}{1-R_{K_1}}
    \right), 
    \label{convergence rate of reconstruction of vector fields}
\end{align}
where $R_{K_1} := \sup_{z \in K_1}\NormCpt{z}{K_0}$, and $C'$ is a constant only depending on $f$, $p$, $K_0$, $K_1$, and $\mu$.
\end{theorem}

In an existing work \cite{ishikawaKoopmanOperatorsIntrinsic2025}, jet extended dynamic mode decomposition (JetEDMD) was introduced in the general framework of rigged reproducing kernel Hilbert spaces with intrinsic observables.
The procedure of constructing the matrix $\widehat{\mathbf{C}}_{m,n,Z}$ is originally proposed there in the case of $E=F$ and $f(p)=p$.
Moreover, convergence of  $\widehat{\mathbf{C}}_{m,n,Z}$ is obtained under strong assumptions on the dynamical system and on the underlying reproducing kernel Hilbert space; see, for instance, the hypotheses as in \cite[Section~6.3 and Theorem~6.13]{ishikawaKoopmanOperatorsIntrinsic2025}.

The present paper is of a different, more algebraic-analytic nature, and should be viewed as a companion to \cite{ishikawaKoopmanOperatorsIntrinsic2025}.
Specialising the reproducing kernel Hilbert space to the Fock space associated with an analytic dynamical system, we work with the push-forward $f_*$ on the space of locally analytic functionals $\mathcal{O}_E(K)'$ and represent it, via the Fourier--Borel transform.
Within this concrete analytic setting, Theorems~\ref{thm: main thm, intro}, \ref{thm: main thm 2, intro} and \ref{thm: main thm 3, intro} provide a convergence theory for data-driven approximations which is substantially more general in the analytic Fock-space case than the one available in \cite{ishikawaKoopmanOperatorsIntrinsic2025}, and which also covers problems that were not treated there.

\medskip

The paper is organized as follows.
In Section~\ref{sec: Fourier--Borel transform} we recall basic facts on holomorphic functions, the Fourier--Borel transform, and spaces of entire functions of exponential type, and we fix notation for spaces of locally analytic functionals.
Section~\ref{sec: fock space} is devoted to the Fock space and its relation to ${\rm Exp}(E;K)$ and $\mathcal{O}_E(K)$.
In Section~\ref{sec: approximations of push-forwards} we introduce the push-forward on locally analytic functionals and establish its finite-dimensional approximations, proving in particular Theorem~\ref{thm: main thm}.
Section~\ref{sec: approximation of analytic map} is concerned with the approximation of analytic maps and least-squares polynomials, and contains the proof of Theorem~\ref{thm: regression} and of our results on the smallest eigenvalues of Hankel matrices.
Finally, in Section~\ref{sec: reconstruction of vector fields} we apply our framework to analytic ordinary differential equations and establish the convergence of the data-driven reconstruction of analytic vector fields from finite flow data, proving Theorem~\ref{thm: reconstruction of vector field}.
In Appendix~\ref{appendix: Asymptotics of the lowest eigenvalues of Hankel matrices} we recall Hirschman's asymptotic formula for the smallest eigenvalues of certain Hankel matrices and adapt it to our setting, and we collect in Section~\ref{sec: notation table} a table of notation used throughout the paper.

\subsection*{Acknowledgement}
The author acknowledges support from JST CREST Grant Numbers JPMJCR1913, JPMJCR24Q6, and JPMJCR24Q1, ACT-X Grant Number JPMJAX2004, and JSPS KAKENHI Grant Numbers JP24K06771, JP24K16950, and JP24K21316.

\section{The Fourier--Borel transform} \label{sec: Fourier--Borel transform}
In this section, we summarize basic facts for the Fourier--Borel transform.
See, for example, \cite{morimotoIntroductionSatosHyperfunctions1993} for more details.
Let $K \subset E$ be a compact convex set.

We define the supporting function $h_K: E \to \mathbb{R}$ by
\begin{align}
    h_K(\zeta) &:= \sup_{z \in K}\mathrm{Re}(\langle \zeta,~z\rangle_E).
\end{align}
We note that by definition, for compact convex sets $K_1, K_2 \subset E$, $h_{K_1} + h_{K_2} = h_{K_1 + K_2}$ holds, where $K_1 + K_2 := \{v_1 + v_2 : v_1 \in K_1,~v_2 \in K_2 \}$, and $h_{K_1} \le h_{K_2}$ holds if $K_1 \subset K_2$.

For a compact convex subset $K \subset E$, we define the Banach space composed of entire functions on $E$ by
\begin{align}
    {\rm Exp}^{\rm b}(E; K) := \left\{h \in \mathcal{O}_{E}(E) : \|h\|_{\exp, K} < \infty \right\},
\end{align}
where $\|\cdot\|_{\exp, K}: \mathcal{O}_{E}(E) \to \mathbb{R}_{\ge 0} \cup \{\infty\}$ is a norm defined by
\begin{align}
    \|g\|_{\exp, K} := \sup_{\zeta \in E}|g(\zeta)|e^{-h_K(\zeta)}.
\end{align}

\begin{definition}
    \label{def: Exp, open}
    For an open convex subset $W \subset E$, we define
    \begin{align}
    {\rm Exp}(E; W) := \bigcup_{K_1 \subset W}{\rm Exp}^{\rm b}(E; K_1),
    \end{align}
    where we regard the right hand side as the inductive limit of $\{{\rm Exp}^{\rm b}(E; K_1) \}_{K_1 \subset W}$, considered as an inductive system via the natural inclusion maps and $K_1$ varies over compact convex subsets.
\end{definition}

\begin{definition}
    \label{def: Exp, compact}
    For a compact convex subset $K \subset E$, we define
    \begin{align}
    {\rm Exp}(E; K) := \bigcap_{K_1 \Supset K}{\rm Exp}^{\rm b}(E; K_1),
    \end{align}
    where we regard the right hand side as the projective limit of $\{{\rm Exp}^{\rm b}(E; K_1) \}_{K_1 \Supset K}$, considered as a projective system via the natural inclusion maps and $K_1$ varies over compact convex subsets.
\end{definition}

\begin{definition}
    For $\mu \in \mathcal{O}_E(E)'$, we define $\overline{\mathbf{FB}}_E(\mu) \in {\rm Exp}(E; E)$ by
    \begin{align}
        \overline{\mathbf{FB}}_E(\mu)(\zeta) := \overline{\mu(\mathbf{e}_\zeta)}.
    \end{align}
\end{definition}
Since the dual map $\mathcal{O}_E(S)' \to \mathcal{O}_E(E)'$ of the restriction map is injective if $S$ is compact or open convex subset of $E$ (see \cite[Theorem 1.6.2]{morimotoIntroductionSatosHyperfunctions1993}),  we naturally regard $\mathcal{O}_E(S)'$ as a subspace of $\mathcal{O}_E(E)'$.
Then, we state a crucial result, which is a slight modification of fact referred to as the Ehrenpreis--Martineau theorem (see, for example, \cite[Theorem 6.4.5]{morimotoIntroductionSatosHyperfunctions1993}):
\begin{theorem}\label{thm: EM theorem}
    Let $S \subset E$ be a compact (or open) convex set.
    The antilinear map $ \overline{\mathbf{FB}}_E$ induces a homeomorphism from $\mathcal{O}_E(S)'$ to ${\rm Exp}(E; S)$.
\end{theorem}

\begin{definition}
Let $S \subset E$ be a compact (or open) convex set.
Let $F := \mathbb{C}^r$ and let  $T \subset F$ be a compact (or open) convex set.
Let $\mathcal{L}: \mathcal{O}_E(S)' \to \mathcal{O}_F(T)'$ be a continuous linear map.
Then, we define the continuous linear map
\begin{align}
    \widetilde{\mathcal{L}} := \overline{\mathbf{FB}}_{F}\circ \mathcal{L} \circ \overline{\mathbf{FB}}_{E}^{-1}: {\rm Exp}(E; S) \to {\rm Exp}(F; T).
\end{align}   
\end{definition}

\begin{remark}\label{rem: continuity on Exp}
Let $K \subset E$ and $L \subset F$ be compact convex sets.
Let $\{K_i\}_{i=1}^\infty$ be a sequence of compact convex sets of $E$ such that $K \Subset K_{i+1} \Subset K_{i}$ for all $i \ge 1$ and $\bigcap_{i\ge 1} K_i = K$.
Then, the topology of the locally convex space ${\rm Exp}(E; K)$ is defined by the set of semi-norms $\{\|\cdot\|_{\exp, K_i}\}_{i=1}^\infty$.
Consider a continuous linear map $\widetilde{\mathcal{L}}: {\rm Exp}(E; K) \to {\rm Exp}(F; L)$.
Then, for any $L_1 \Supset L$, there exist $i \ge 1$ and $C>0$ such that $\|\widetilde{\mathcal{L}}(h)\|_{\exp, L_1} \le C \|h\|_{\exp, K_i}$ holds for any $h \in {\rm Exp}(E; K)$ (see \cite[Proposition 5, \S 1, Chapter II]{bourbakiTopologicalVectorSpaces1987}).
\end{remark}

For $p \in E$ and $\alpha = (\alpha_1, \dots, \alpha_d) \in \mathbb{Z}_{\ge 0}^d$, we define $\delta_{p}^{(\alpha)} \in \mathcal{O}_E(\{p\})'$ by
\begin{align}
    \delta_{p}^{(\alpha)}(h) := (\partial_{z_1}^{\alpha_1} \cdots \partial_{z_d}^{\alpha_d}h)(p)
\end{align}
for $h \in \mathcal{O}_E(\{p\})$.
Then, we define
\begin{align}
    \mathfrak{D}_{p,n} := {\rm span}\left(
    \left\{\delta_{p}^{(\alpha)} : \alpha=(\alpha_1,\dots, \alpha_d) \in \mathbb{Z}_{\ge 0}^d,~\sum_{i=1}^d \alpha_i \le n\right\}\right)
\end{align}
for each $n \in \mathbb{Z}_{\ge 0}$.
Let $\mathfrak{D}_p := \cup_{n \ge 0} \mathfrak{D}_{p,n}$.

\begin{remark}
    \label{rem: jets}
    Let $\mathfrak{m}_p := \{h \in \mathcal{O}_E(\{p\}) : h(p) = 0 \}$ be an ideal of $\mathcal{O}_E(\{p\})$.
    Then, we see that the image of the dual map $ (\mathcal{O}_E(\{p\}) / \mathfrak{m}_p^{n+1})' \to \mathcal{O}_E(E)'$ of the natural surjection coincides with $\mathfrak{D}_{p,n}$.
    We also remark that the set $\mathcal{O}_E(\{p\}) / \mathfrak{m}_p^{n+1}$ is a geometrically canonical object, usually referred to as the space of $n$-jets at $p$ (see \cite[Remark 12.7]{kolarNaturalOperationsDifferential1993}).
\end{remark}

We define a finite-dimensional subspace of ${\rm Exp}(E; E)$ by
\begin{align}
    \widetilde{\mathfrak{D}}_{p,n} := \overline{\mathbf{FB}}_E(\mathfrak{D}_{p,n}).
\end{align}
For $n \in \mathbb{Z}_{\ge 0}$, let $\mathcal{P}_n(E) \subset {\rm Exp}(E; E)$ be the space of polynomial functions on $E$ of degree at most $n$.
Then, by definition, we immediately have the following proposition:
\begin{proposition}\label{prop: basic properties of FB}
    For $p \in E$ and $\alpha = (\alpha_1,\dots, \alpha_d) \in \mathbb{Z}_{\ge 0}^d$, we have $\overline{\mathbf{FB}}_E(\delta_p^{(\alpha)})(z)= z^{\alpha} \mathbf{e}_{p}(z)$.
    In particular, we have $\widetilde{\mathfrak{D}}_{p,n} = \left\{ P\mathbf{e}_{p} : P \in \mathcal{P}_n(E)\right\}$.
\end{proposition}

\section{The Fock space} \label{sec: fock space}
We introduce a Hilbert space contained in $\mathcal{O}_E(E)$ as follows:
\begin{definition}[Fock space]
We define 
\begin{align*}
H_E  := \left\{ h \in \mathcal{O}_E(E) : \int_{E} |h(z)|^2 e^{-\|z\|_E^2}\,{\rm d}\mu_E(z) < \infty\right\}
\end{align*}
with the inner product
\begin{align*}
    \langle h,\,g\rangle_{H_E} &:= \int_{E} h(z) \overline{g(z)}  e^{-\|z\|_E^2}\,{\rm d}\mu_E(z).
\end{align*}
\end{definition}
This Hilbert space $H_E$ is usually called the Fock space or the Segal--Bargmann space.
For details of the theory of the Fock space, see, for example, \cite{zhuAnalysisFockSpaces2012}.
We have the following proposition:
\begin{proposition}\label{prop: RKHS}
For any $w \in E$, we have $\mathbf{e}_w \in H_E$  and $\langle h, \mathbf{e}_w\rangle_{H_E} = h(w)$.
\end{proposition}
\begin{proof}
It follows from \cite[Section 2.1]{zhuAnalysisFockSpaces2012}.
\end{proof}
\begin{remark}
Let $k(z,w) := e^{\langle z, w\rangle_E} =  \mathbf{e}_w(z)$ for $z,w \in E$.
Let $k_w = \mathbf{e}_w$.
The two-variable function $k$ is usually referred to as the positive definite kernel associated with $H_E$.
This Hilbert space is also characterized as a Hilbert space $H$ composed of $\mathbb{C}$-valued functions on $E$  satisfying both $k_w \in H$ and $\langle h, k_w \rangle_{H} = h(w)$ for any $w \in E$.
Such a Hilbert space is usually referred to as the reproducing kernel Hilbert space (RKHS).
For the general theory of RKHS, see, for example, \cite{saitohTheoryReproducingKernels2016}.
\end{remark}

The relationship between the spaces of entire functions and the Fock space is described as in the following proposition:
\begin{proposition}\label{prop: gelfand triple}
    Let $K \subset E$ be a compact convex set.
    \begin{enumerate}[label = (\arabic*)]
        \item \label{continuity i)}
        For $h \in {\rm Exp}(E; K)$ and compact convex set $K_1 \Supset K$, we have
        \begin{align*}
            \|h\|_{H_E} \le C_{K_1} \|h\|_{\exp, K_1},
        \end{align*}
        where $C_{K_1} := \sqrt{\int_{E} e^{2h_{K_1}(z)-\|z\|_E^2}\,{\rm d}\mu_E(z)}$.
        In particular, the inclusion ${\rm Exp}(E; K) \to H_E$ is continuous.
        \item \label{continuity ii)}
        For $h \in H_E$ and a compact subset $K_1 \subset E$, we have
        \begin{align*}
            \sup_{z \in K_1} |h(z)| \le \sup_{z\in K_1}e^{\|z\|_E^2/2} \cdot \|h\|_{H_E}.
        \end{align*}
        In particular, the natural map $H_E \to \mathcal{O}_E(K);~h \mapsto h|_K$ is continuous.
        \item \label{compatibility of pairing}
        For any $g \in H_E$ and $h \in {\rm Exp}(E; K)$, we have
    \begin{align}
        \langle g, h \rangle_{H_E} = \overline{\mathbf{FB}}_E^{-1}(h)(g).
    \end{align}
    \end{enumerate}
\end{proposition}
\begin{proof}
First, we prove \ref{continuity i)}.
Let $K_1 \subset E$ be a compact convex subset.
By definition of $H_E$, the inequality is proved as follows:
\begin{align*}
    \|h\|_{H_E}^2 &= 
    \int_{E} |h(z)|^2 e^{-\|z\|_E^2} {\rm d}\mu_E(z) \\
            &\le \|h\|_{\exp, K_1}^2 \cdot \int_{E} e^{2h_{K_1}(z)} e^{-\|z\|_E^2} {\rm d}\mu_E(z)
            = C_{K_1}^2 \|h \|_{\exp, K_1}^2.
\end{align*}
Next, we prove \ref{continuity ii)}.
Let $h \in H_E$.
By Proposition \ref{prop: RKHS}, we have $\langle h,  \mathbf{e}_z\rangle_{H_E} = h(z)$ for $z \in E$.
Thus, by the Cauchy-Schwarz inequality, $|h(z)| \le \|\mathbf{e}_z\|_{H_E} \|h\|_{H_E}$.
Since $\|\mathbf{e}_z\|_{H_E}^2 = \mathbf{e}_z(z) = e^{\|z\|_E^2}$ by Proposition \ref{prop: RKHS} again, we have \ref{continuity ii)}.
Regarding \ref{compatibility of pairing}, it suffices to show the identity $\langle h, \overline{\mathbf{FB}}_E(\mu)\rangle_{H_E} = \mu(h)$ for $\mu \in \mathcal{O}_E(K)'$ by Theorem \ref{thm: EM theorem}.
Let $\mu_0$ be the compactly supported complex-valued measure on $E$ such that $\langle \mu \,|\, h \rangle = \int h \,{\rm d}\mu_0$.
Since $\overline{\mathbf{FB}}_E(\mu)(z)  = \overline{\int \mathbf{e}_z(w) \,{\rm d}\mu_0(w)}$, we obtain \ref{compatibility of pairing} using Proposition \ref{prop: RKHS} and Fubini's theorem as follows:
\begin{align*}
    \langle h, \overline{\mathbf{FB}}_E(\mu)\rangle_{H_E}
    &=\int_{E} \int  \mathbf{e}_z(w) \,{\rm d}\mu_0(w)h(z)  e^{- \|z\|_E^2} {\rm d}\mu_E\\
    &= \int \langle h, \mathbf{e}_w\rangle_{H_E} \,{\rm d}\mu_0(w) 
    = \int h(w) \,{\rm d}\mu_0(w) = \langle \mu \,|\, h \rangle.
\end{align*}
\end{proof}
\begin{remark}
    Proposition \ref{prop: gelfand triple} implies that the triplet $({\rm Exp}(E; K), H_E, \mathcal{O}_E(K))$ constitutes the Gel'fand triple.
    For the notion of the Gel'fand triple, see, for example, \cite{bohmDiracKetsGamow1989}.
\end{remark}

We say that a subset $S$ of a linear space is absolutely convex if $ax + by \in S$ for any $x, y \in S$ and $a,b \in \mathbb{C}$ with $|a| + |b| \le 1$.
We introduce a norm on $E$ induced by a compact absolutely convex neighborhood of $0$:
\begin{definition}
For a compact absolutely convex neighborhood $K \subset E$ of $0$, we define the norm $|\cdot|_K: E \to \mathbb{R}_{\ge 0}$ by
\begin{align}
    \NormCpt{z}{K} := \inf\{ r > 0 : z \in rK\}.
\end{align}
\end{definition}
For example, $\NormCpt{z}{K} = \|z\|_E/r$ (resp. $\max_{i=1,\dots, d}|z_i|/r_i$) if $K = \{z \in E : \|z\|_E \le r\}$ (resp. $\{z \in E : |z_i| \le r_i \text{ for }i=1, \dots, d\}$).
We note that $z \in K$ if and only if $|z|_K \le 1$.

For a subset $S \subset E$, we define
\begin{align}
    S^\circ := \left\{z \in E :  \sup_{w \in S}|\langle z, w \rangle| \le 1 \right\}.
\end{align}
\begin{lemma} \label{lem: supporting functions of compact barreled sets}
    Let $K \subset E$ be a compact absolutely convex neighborhood of $0$.
    Then, $K^\circ \subset E$ is also a compact absolutely convex neighborhood of $0$, and $h_K(\zeta) = \NormCpt{\zeta}{K^\circ}$ holds for $\zeta \in E$.
    In particular, $|\langle \zeta, z\rangle_E| \le \NormCpt{\zeta}{K^\circ} \NormCpt{z}{K}$ for $z, \zeta \in E$.
\end{lemma}
\begin{proof}
    The first statement is obvious by definition.
    As for the second statement, by definition of $|\cdot|_{K^\circ}$, we see that $\zeta \in rK^\circ$ if and only if $\sup_{z \in K}|\langle \zeta,z\rangle_E| \le r$,  namely, $\sup_{z \in K}|\langle \zeta, z \rangle_E| = \NormCpt{\zeta}{K^\circ}$.
    Since $K$ is absolutely convex, we have $h_K(\zeta) = \sup_{z \in K}{\rm Re}(\langle \zeta, z \rangle_E) = \sup_{z \in K}|\langle \zeta, z \rangle_E| = \NormCpt{\zeta}{K^\circ}$.
\end{proof}
For $p \in K$, we denote the orthogonal projection to $\widetilde{\mathfrak{D}}_{p,n}$ by
\begin{align}
    \pi_{p,n} : H_E \to \widetilde{\mathfrak{D}}_{p,n}.
\end{align}
Although the orthogonal projection $\pi_{p,n}$ does not behave well 
on ${\rm Exp}^{\rm b}(E; K)$, we have an explicit estimate of the change in the norms of ${\rm Exp}^{\rm b}(E; K)$ by the orthogonal projection as follows:
\begin{lemma}\label{lem: proj in Exp}
    Let $p \in E$ and let $K_0$ be a compact absolutely convex neighborhood of $0$.
    Let $K := p + K_0$.
    Let $w \in K$ and define $r_w := \NormCpt{w-p}{K_0} \le 1$.
    Then, we have
    \begin{align}
        \|\mathbf{e}_w - \pi_{p,n}\mathbf{e}_w\|_{\exp, K} 
        \le \frac{e^{(1 + r_w)\|p\|_{K_0^\circ}}}{\sqrt{2\pi}(1 - r_w)}\cdot 
        \sqrt{n+1} r_w^{n+1}.
        \label{convergence rate}
    \end{align}
    Here, we interpret the right hand side as $+\infty$ when $r_w=1$.
\end{lemma}
\begin{proof}
    Since $\mathbf{e}_{w-p}(z) = \sum_{m=0}^\infty \frac{\langle z-p, w-p \rangle_E^m }{m!}\mathbf{e}_{w-p}(p)$, we have
    \begin{align*}
        \mathbf{e}_{w}(z) = \mathbf{e}_{w-p}(p) \mathbf{e}_p(z) \sum_{m=0}^\infty \frac{\langle z-p, w-p \rangle_E^m}{m!} .
    \end{align*} 
    Combining this with Proposition \ref{prop: basic properties of FB}, we have
    \begin{align}
        \pi_{p,n}\mathbf{e}_w(z) = \mathbf{e}_{w-p}(p)\mathbf{e}_p(z)\sum_{m=0}^n \frac{\langle z-p,w-p\rangle_{E}^m}{m!}.
    \end{align}
    Since $\mathbf{e}_w(z) = \mathbf{e}_{w-p}(p)\mathbf{e}_p(z)\mathbf{e}_{w-p}(z-p)$, by the Taylor theorem \cite[Theorem 7.6]{apostolCalculusVolOnevariable1967}, we have
    \begin{align*}
        &\mathbf{e}_w(z) - \pi_{p,n}\mathbf{e}_w(z) \\
        &= \mathbf{e}_{w-p}(p) \mathbf{e}_p(z)\cdot \frac{\langle z-p,w-p\rangle_E^{n+1}}{n!} \int_0^1 (1-t)^n e^{t \langle z-p,w-p\rangle_E} {\rm d}t.
    \end{align*}
    Thus, we have
    \begin{align*}
        &\left| \mathbf{e}_w(z) - \pi_{p,n}\mathbf{e}_w(z)\right| e^{-{\rm Re}(\langle z,p \rangle_E) - \NormCpt{z}{K_0^\circ}} \\
        &\le  |\mathbf{e}_{w-p}(p)|  \frac{r_w^{n+1}|z-p|_{K_0^\circ}^{n+1}}{n!}\int_0^1 (1-t)^n e^{t r_w |z-p|_{K_0^\circ} -|z|_{K_0^\circ}} {\rm d}t  \\
        &\le |\mathbf{e}_{w-p}(p)|e^{|p|_{K_0^\circ}} \frac{r_w^{n+1}|z-p|_{K_0^\circ}^{n+1}}{n!}\int_0^1 (1-t)^n e^{-(1-tr_w) |z-p|_{K_0^\circ}} {\rm d}t \\
        &\le e^{(1 + r_w)|p|_{K_0^\circ}} \frac{\sqrt{n+1}}{\sqrt{2\pi}} r_w^{n+1} \int_0^1 (1-t)^n \frac{{\rm d}t}{(1-r_wt)^{n+1}},
    \end{align*}
    where we use Lemma \ref{lem: supporting functions of compact barreled sets} in the first inequality, the triangle inequality $|z|_{K_0^\circ} \ge |z-p|_{K_0^\circ} - |p|_{K_0^\circ}$ in the second inequality, and
    \[s^{n+1} e^{-bs} \le \left( \frac{n+1}{b}\right)^{n+1} e^{-n-1} \le \frac{\sqrt{n+1} n!}{\sqrt{2\pi}b^{n+1}}\]
    for $s := |z-p|_{K_0^\circ} \ge 0$ and $b := (1-r_wt) > 0$ in the third inequality (here, we also use the Stirling formula).
    By the H\"older inequality, we have
    \begin{align*}
       &\int_0^1 (1-t)^n \frac{{\rm d}t}{(1-r_wt)^{n+1}}\\
       &\le \left(\int_0^1 \frac{{\rm d}t}{(1-r_wt)^{n+1}}\right)^{1/(n+1)} \left(\int_0^1 (1-t)^{n+1} \frac{{\rm d}t}{(1-r_wt)^{n+1}}\right)^{n/(n+1)}\\
       &\le \frac{1}{1-r_w}.
    \end{align*}
    Therefore, we obtain the inequality \eqref{convergence rate}.
\end{proof}

We define
    \begin{align}
        u_{p,\alpha}(z) := \frac{e^{-\|p\|_E^2/2}}{\sqrt{\alpha!}}(z - p)^{\alpha} \mathbf{e}_p(z).
    \end{align}
Then, we have the following proposition:
\begin{proposition}\label{prop: ONS of H}
    Let $p \in E$.
    Then, $\{u_{p, {\alpha}}\}_{\alpha \in \mathbb{Z}_{\ge 0}^d}$ is an orthonormal basis of $H_E$.
\end{proposition}
\begin{proof}
    Since
    \begin{align*}
         \langle u_{p, {\alpha}}, u_{p,\beta}\rangle_{H_E} 
         &= e^{-\|p\|_E^2}\int_E \frac{\overline{(z - p)^{\alpha}}}{\sqrt{{\alpha}!}} \frac{(z - p)^{\beta}}{\sqrt{\beta!}} e^{\langle p,z\rangle_E + \langle z, p\rangle_E} e^{-\|z \|_E^2} {\rm d}\mu_E(z)\\
         &= \int_E \frac{\overline{(z - p)^{\alpha}}}{\sqrt{{\alpha}!}} \frac{(z - p)^{\beta}}{\sqrt{\beta!}}  e^{-\|z-p\|_E^2} {\rm d}\mu_E(z)\\
         &= \langle u_{0, \alpha}, u_{0, \beta} \rangle_{H_E},
    \end{align*}
    we may assume $p=0$.
    By Proposition \ref{prop: basic properties of FB} and \ref{compatibility of pairing} in Proposition \ref{prop: gelfand triple}, we have
    \begin{align*}
        \langle u_{0, \alpha}, u_{0,\beta} \rangle_{H_E} 
        = \overline{\mathbf{FB}}_E^{-1}(u_{0,{\beta}})(u_{0, \alpha})
        =\frac{ \delta_0^{(\beta)} (z^\alpha)}{\sqrt{{\beta}! \alpha!}} 
        = 
        \begin{cases}
            1 & \text{ if $\alpha = \beta$} \\
            0 & \text{otherwise}
        \end{cases}
    \end{align*}
    Thus, $\{ u_{0, \alpha} \}_{{\alpha}}$ is an orthonormal system.
    Since the normalized monomials form an orthonormal basis of the Fock space, $\{ u_{0,\alpha}\}_{\alpha}$ is complete.
    The case of general $p$ follows by the unitary translation used above.
\end{proof}

\section{Finite-dimensional approximations of push-forwards} \label{sec: approximations of push-forwards}

Let us define the pull-back and push-forward in an accurate manner.
Suppose that there exists an open subset $U\supset K$ of $E$ and a holomorphic map $\mathbf{f}:U \to F$ such that $\mathbf{f}|_{K} = f$. 
As $\mathcal{O}_E(S)$ for a compact set $S$ is the inductive limit of the $\mathcal{O}_E(W)$ over open neighborhoods $W$ of $S$ with restriction maps, we denote the equivalence class of $h \in \mathcal{O}_E(W)$ in $\mathcal{O}_E(S)$ by $[h,W]$.
The pull-back $f^*: \mathcal{O}_F(f(K)) \to \mathcal{O}_E(K);~[h, V]  \mapsto [h \circ \mathbf{f}, \mathbf{f}^{-1}(V)]$ is well-defined and also continuous as in the following proposition:
\begin{proposition}
    The pull-back $f^*: \mathcal{O}_F(f(K)) \to \mathcal{O}_E(K)$ is continuous.
\end{proposition}
\begin{proof}
    Let $U \supset K$ and $W \supset f(K)$ be open subsets such that there exists a holomorphic map $\mathbf{f}: U \to W$ with $\mathbf{f}|_K = f$.
    It suffices to show that $\mathcal{O}_F(W) \to \mathcal{O}_E(U); h \mapsto h \circ \mathbf{f}$ is continuous.
    We note that the topology of $\mathcal{O}_E(U)$ is defined by the seminorms $\sup_{z \in K_1}|h(z)|$ for the compact subsets $K_1 \subset U$, and the topology of $\mathcal{O}_F(W)$ is also defined in the same way.
    Thus, the continuity follows from the following formula: $\sup_{z \in K_1}|\mathbf{f}^*(h)(z)| = \sup_{w \in \mathbf{f}(K_1)}|h(w)|$.
\end{proof}
For any locally closed set $S \supset f(K)$, we define the pull-back $f^*: \mathcal{O}_F(S) \to \mathcal{O}_E(K)$ by the composition of the restriction map $\mathcal{O}_F(S) \to \mathcal{O}_F(f(K))$ with the pull-back $f^*: \mathcal{O}_F(f(K)) \to \mathcal{O}_E(K)$.
Then, we define the push-forward:
\begin{definition}[Push-forward]
For an analytic map $f: K \to F$, we define the {\em push-forward} $f_*: \mathcal{O}_E(K)' \to \mathcal{O}_F(F)' $ by the dual map $(f^*)'$ of the pull-back $f^*: \mathcal{O}_F(F) \to \mathcal{O}_E(K);~h \mapsto h \circ f$.
\end{definition}
Then, we have the following proposition:
\begin{proposition} \label{prop: basic properties of f_*}
    Let $p \in K$.
    Then, we have
    \begin{align}
    \widetilde{f_*}\left(\widetilde{\mathfrak{D}}_{p,n}\right) \subset \widetilde{\mathfrak{D}}_{f(p),n}
    \end{align}
    for all $n \ge 0$.
    Moreover, for any $P \in \mathcal{P}_n(E)$, we have
    \begin{align*}
    \widetilde{f_*}\left(P\mathbf{e}_p\right)
    \equiv
    \left(P\circ({\rm d}f_p)^*\right)\mathbf{e}_{f(p)}
    \mod \widetilde{\mathfrak{D}}_{f(p),n-1},
    \end{align*}
    where $({\rm d}f_p)^*:F\to E$ denotes the Hermitian adjoint of ${\rm d}f_p:E\to F$.
\end{proposition}
\begin{proof}
    It follows from \cite[Lemma 2.2]{ishikawaBoundedCompositionOperators2023}.
\end{proof}
As a direct consequence, we have the following statement:
\begin{corollary}
    Assume that $E=F$ and $f(p)=p$.
    Let $\lambda_1,\dots,\lambda_d$ be the eigenvalues of the Fr\'echet derivative of $f:E\to E$ at $p$.
    Then, $\widetilde{f_*}$ induces a linear transform of $\widetilde{\mathfrak{D}}_{p,n}$.
    Moreover, the set of eigenvalues of $\widetilde{f_*}|_{\widetilde{\mathfrak{D}}_{p,n}}$ coincides with
    \begin{align*}
    \left\{
    \overline{\lambda_1}^{\alpha_1}\cdots
    \overline{\lambda_d}^{\alpha_d}
    :
    \alpha = (\alpha_1,\dots,\alpha_d) \in \mathbb{Z}_{\ge 0}^d,\ 
    |\alpha| \le n
    \right\}.
    \end{align*}
\end{corollary}

We fix a numbering $ \{\alpha^{(i)}\}_{i=1}^\infty = \mathbb{Z}_{\ge 0}^d$ of $\mathbb{Z}_{\ge 0}^d$ such that $\alpha^{(i+1)}$ is the $i$-th elementary vector for $i=1,\dots, d$, and that $|\alpha^{(i)}| \le |\alpha^{(j)}|$ if $i \le j$.
We define $u_{p,i} := u_{p, \alpha^{(i)}}$.
We also fix a numbering of $\{\beta^{(i)}\}_{i=1}^\infty = \mathbb{Z}_{\ge 0}^r$ in the same manner.
For $q \in F$ and $\beta \in \mathbb{Z}_{\ge 0}^r$, we define
\begin{align}
    v_{q,\beta}(w) := \frac{e^{-\|q\|_F^2/2}}{\sqrt{\beta!}}(w - q)^{\beta} \mathbf{e}_q(w),
\end{align}
and $v_{q,i} := v_{q, \beta^{(i)}}$.

For $n \ge 0$, $z \in E$, and $w \in F$, we define (horizontal) vectors of $\mathbb{C}^{1 \times r_n^{E}}$ and $\mathbb{C}^{1 \times r_n^{F}}$ by
\begin{align}
    \mathbf{u}_{p, n}(z) &:= \left(u_{p,1}(z), u_{p,2}(z), \dots, u_{p,r_n^{E}}(z)\right), \\
    \mathbf{v}_{q,n}(w) &:= \left(v_{q,1}(w), v_{q,2}(w), \dots, v_{q,r_n^{F}}(w)\right).
\end{align}
For $n \ge 0$, $p \in E$, $q \in F$, $Z = (z^1,\dots,z^N) \in E^N$, and $W = (w^1,\dots, w^N) \in F^N$, we define matrices $\mathbf{U}_{p,n,Z} \in \mathbb{C}^{N \times r_n^{E}}$  and $\mathbf{V}_{q,n,W} \in \mathbb{C}^{N \times r_n^{F}}$ by
\begin{align}
\mathbf{U}_{p, n, Z}  := 
\begin{pmatrix}
    \mathbf{u}_{p,n}(z^1)\\
    \vdots\\
    \mathbf{u}_{p,n}(z^N)
\end{pmatrix},~~~ 
\mathbf{V}_{q,n, W} :=
\begin{pmatrix}
    \mathbf{v}_{q,n}(w^1)\\
    \vdots\\
    \mathbf{v}_{q,n}(w^N)
\end{pmatrix}.
\end{align}

\begin{definition}[Finite-dimensional approximation of $f_*$]
Let $p \in E$.
Let $Z = (z^1,\dots,z^N) \in K^N$ and let $W := (f(z^1),\dots,f(z^N))$.
For integers $m$ and $n$ with $m \le n$, we define $\widehat{\mathbf{C}}_{m,n,Z} \in \mathbb{C}^{r_m^{F} \times r_m^{E}}$ by the leftmost submatrix of $\mathbf{V}_{f(p),m,W}^* (\mathbf{U}_{p,n,Z}^*)^\dagger$ of size $r_m^{F} \times r_m^{E}$, namely,
\begin{align}
\mathbf{V}_{f(p),m,W}^* (\mathbf{U}_{p,n,Z}^*)^\dagger
=
\begin{pmatrix}
\widehat{\mathbf{C}}_{m,n,Z}  & * 
\end{pmatrix}. \label{C_mnN}
\end{align}
\end{definition}

First, we prove a key lemma that provides an estimate of the error between a representation matrix of the transformed push-forward $\widetilde{f_*}$ and the matrix $\widehat{\mathbf{C}}_{m,n,Z}$:
\begin{lemma}\label{lem: main lemma}
    Let $p \in E$ and let $K_0 \subset E$ be a compact absolutely convex neighborhood of $0$.
    Let $K := p + K_0$.
    Let $f: K \to F$ be an analytic map.
    Let $m, n \in \mathbb{Z}_{\ge 0}$ with $m \le n$.
    For each $k \in \mathbb{Z}_{\ge 0}$, let $\mathbf{C}_k$ be the representation matrix of $\widetilde{f_*}|_{\widetilde{\mathfrak{D}}_{p,k}}: \widetilde{\mathfrak{D}}_{p,k} \to \widetilde{\mathfrak{D}}_{f(p),k}$ with respect to the bases $(u_{p,i})_{i=1}^{r_k^E}$ and $(v_{f(p),i})_{i=1}^{r_k^F}$. Let $z_0^1,\dots, z_0^N \in K_0$.
    We define $\widehat{\mu}_N := N^{-1} \sum_{i=1}^N\delta_{z_0^i}$.
    Let $Z_N := (p + z_0^1,\dots, p + z_0^N) \in K^N$.
    Assume that $\mathbf{D}_n(\widehat{\mu}_N)$ is invertible.
    Then, we have
    \begin{align}
        \left\| \mathbf{C}_m - \widehat{\mathbf{C}}_{m,n,Z_N}\right\|_{\rm Fr} \le C 
        \frac{\sqrt{m!}}{\Lambda_n(\widehat{\mu}_N)^{1/2}}
            \sqrt{\int \NormCpt{z}{K_0}^{2n} {\rm d}\widehat{\mu}_N(z)},
        \label{main theorem inequality}
     \end{align}
     where $C > 0$ is a constant only depending on $f$, $p$, $K_0$. 
\end{lemma}
\begin{proof}
    Let $z^i := p + z_0^i$ and let $W_N := (f(z^1),\dots, f(z^N))$.
    Let $\mathbf{P}_m^E :=\begin{pmatrix} \mathbf{I}_{r_m^{E}} &\mathbf{O} \end{pmatrix} \in \mathbb{C}^{r_m^E \times r_n^E}$ and $\mathbf{P}_m^F :=\begin{pmatrix} \mathbf{I}_{r_m^{F}} &\mathbf{O} \end{pmatrix} \in \mathbb{C}^{r_m^F \times r_n^F}$.
    Then, we have
    \begin{align*}
        \widehat{\mathbf{C}}_{m,n,Z_N}
        &=
        \mathbf{P}_m^F
        \mathbf{V}_{f(p),n, W_N}^* (\mathbf{U}_{p,n, Z_N}^*)^\dagger
        (\mathbf{P}_m^E)^\top,\\
        \mathbf{C}_m
        &=
        \mathbf{P}_m^F
        \mathbf{C}_{n}
        (\mathbf{P}_m^E)^\top.
    \end{align*}
    For $n \ge 0$, we define
    \begin{align}
        \mathbf{F}_n
        &:=
        \begin{pmatrix}
             \sqrt{\alpha^{(1)}!} & &\\
            & \ddots & \\
            & & \sqrt{\alpha^{(r^E_n)}!}
        \end{pmatrix}
    \end{align}
    and let $\nu$ be the measure defined by $\int h {\rm d}\nu = e^{\|p\|_E^2}\int h(x) e^{2\mathrm{Re}(p^*x)} {\rm d}\widehat{\mu}_N(x)$.
    Since $N^{-1}\mathbf{U}_{p,n, Z_N}^*\mathbf{U}_{p,n, Z_N} = \mathbf{F}_n^{-1} \mathbf{D}_n(\nu) \mathbf{F}_n^{-1}$ holds and $\mathbf{D}_n(\widehat{\mu}_N)$ is invertible, we see that $\mathbf{U}_{p,n, Z_N}^*$ is full row-rank, hence $\mathbf{U}_{p,n, Z_N}^* (\mathbf{U}_{p,n, Z_N}^*)^\dagger = \mathbf{I}_{r_n^{E}}$.
    
    Thus, we have
    \begin{align}
        \left\| \mathbf{C}_m - \widehat{\mathbf{C}}_{m,n,Z_N}\right\|_{\rm Fr}
        \le 
        \big\| \mathbf{P}_m^F \mathbf{C}_n \mathbf{U}_{p,n, Z_N}^* -
        \mathbf{P}_m^F\mathbf{V}_{f(p),n, W_N}^*
        \big\|_{\rm Fr} 
        \cdot 
        \left\|(\mathbf{U}_{p,n, Z_N}^*)^\dagger 
        (\mathbf{P}_m^E)^\top \right\|_{\rm op}. \label{target}
    \end{align}
    First, we estimate $\big\| \mathbf{P}_m^F \mathbf{C}_n \mathbf{U}_{p,n, Z_N}^* -
        \mathbf{P}_m^F\mathbf{V}_{f(p),n, W_N}^*
        \big\|_{\rm Fr}$ in \eqref{target}. 
    By Proposition \ref{prop: RKHS}, we see that $\mathbf{e}_z = \sum_{i=1}^\infty \overline{ u_{p,i}(z)} u_{p,i}$ and 
    \begin{align*}
        (\pi_{p,n}\mathbf{e}_{z^1}, \dots, \pi_{p,n}\mathbf{e}_{z^N}) = (u_{p,1},\dots, u_{p,r_n^{E}})\mathbf{U}_{p,n, Z_N}^*.
    \end{align*}
    Thus, we have
    \begin{align*}
        (\pi_{f(p),m}\widetilde{f_*}\pi_{p,n}\mathbf{e}_{z^1}, \dots, \pi_{f(p),m}\widetilde{f_*}\pi_{p,n}\mathbf{e}_{z^N}) 
        &= (v_{f(p),1},\dots, v_{f(p),r_m^{F}})\mathbf{P}_m^F \mathbf{C}_n \mathbf{U}_{p,n, Z_N}^*.
    \end{align*}
    On the other hand, since $\widetilde{f_*}\mathbf{e}_z = \mathbf{e}_{f(z)}$, we have
    \begin{align*}
        (\pi_{f(p),m}\widetilde{f_*}\mathbf{e}_{z^1}, \dots, \pi_{f(p),m}\widetilde{f_*}\mathbf{e}_{z^N})
        &= (v_{f(p),1},\dots, v_{f(p),r_m^{F}})\mathbf{P}_m^F \mathbf{V}_{f(p),n, W_N}^*.
    \end{align*}
    Therefore, we obtain the following identity:
    \begin{align}
        \big\| \mathbf{P}_m^F \mathbf{C}_n \mathbf{U}_{p,n, Z_N}^* -
        \mathbf{P}_m^F\mathbf{V}_{f(p), n, W_N}^*
        \big\|_{\rm Fr}^2
        = \sum_{i=1}^N \left\|\pi_{f(p),m}\widetilde{f_*}(\mathbf{e}_{z^i} - \pi_{p,n}\mathbf{e}_{z^i}) \right\|_{H_F}^2. \label{esti. 1}
    \end{align}
    Let $L$ be the convex hull of $f(K)$ and fix a compact convex set $L_+ \subset F$ such that $L_+ \Supset L$.
    By combining the fact that the operator norm of $\pi_{f(p),m}$ is $1$ with the statement \ref{continuity i)} in Proposition \ref{prop: gelfand triple}, we have
    \begin{align}
        \left\|\pi_{f(p),m}\widetilde{f_*}(\mathbf{e}_{z^i} - \pi_{p,n}\mathbf{e}_{z^i}) \right\|_{H_F}^2 
        &\lesssim 
        \left\|\widetilde{f_*}(\mathbf{e}_{z^i} - \pi_{p,n}\mathbf{e}_{z^i}) \right\|_{\exp, L_+}^2
        \label{conti pi}
    \end{align}
    By continuity of $\widetilde{f_*}: {\rm Exp}(E; K) \to {\rm Exp}(F; L)$ (see Remark \ref{rem: continuity on Exp}), there exists a compact absolutely convex set $K_1 \subset E$ with $K_0 \Subset K_1$ such that
    \begin{align}
        \left\|\widetilde{f_*}(\mathbf{e}_{z^i} - \pi_{p,n}\mathbf{e}_{z^i}) \right\|_{\exp, L_+}^2 
        \lesssim  
        \left\| \mathbf{e}_{z^i} - \pi_{p,n}\mathbf{e}_{z^i} \right\|_{\exp, p + K_1}^2.
        \label{conti. f_*}
    \end{align}
    Since $K_0 \Subset K_1$, there exists $\theta \in (0,1)$ such that
    \[
        \NormCpt{x}{K_1} \le \theta \NormCpt{x}{K_0}
        \qquad (x\in E).
    \]
    Hence $n\theta^{2n}$ is uniformly bounded in $n$, and Lemma \ref{lem: proj in Exp} gives
    \begin{align}
        \left\| \mathbf{e}_{z^i} - \pi_{p,n}\mathbf{e}_{z^i} \right\|_{\exp, p+K_1}^2     
        \lesssim
        n \NormCpt{z^i-p}{K_1}^{2n}
        \lesssim
        \NormCpt{z^i-p}{K_0}^{2n}. \label{inequalities with constants}
    \end{align}
    Combining this with \eqref{esti. 1}, we obtain 
    \begin{align*}
        &\big\| \mathbf{P}_m^F \mathbf{C}_n \mathbf{U}_{p,n, Z_N}^* -
        \mathbf{P}_m^F\mathbf{V}_{f(p),n, W_N}^*
        \big\|_{\rm Fr}
        \cdot
        \big\|(\mathbf{U}_{p,n, Z_N}^*)^\dagger 
        (\mathbf{P}_m^E)^\top \big\|_{\rm op}\notag\\
        &\lesssim
        \sqrt{N}\big\|(\mathbf{U}_{p,n, Z_N}^*)^\dagger 
        (\mathbf{P}_m^E)^\top \big\|_{\rm op}
         \sqrt{\frac{1}{N}\sum_{i=1}^N \NormCpt{z^i-p}{K_0}^{2n}}.
    \end{align*}
    By \cite{grevilleNoteGeneralizedInverse1966}, we have
    \begin{align*}
        \big((\mathbf{U}_{p,n, Z_N}^*)^\dagger\big)^*(\mathbf{U}_{p,n, Z_N}^*)^\dagger 
        = (\mathbf{U}_{p,n, Z_N}^*\mathbf{U}_{p,n, Z_N})^{-1}
        = \frac{1}{N} \mathbf{F}_n \mathbf{D}_n(\nu)^{-1} \mathbf{F}_n.
    \end{align*} 
    Since $\mathbf{D}_n(\widehat{\mu}_N) \lesssim \mathbf{D}_n(\nu)$ and the constants in the above $\lesssim$'s only depending on $f$, $p$, and $K_0$, we obtain the formula \eqref{main theorem inequality}.
\end{proof}

Now, we state our main theorem.
\begin{theorem}\label{thm: main thm}
    Let $p \in E \cap \mathbb{R}^d$ and let $K_0 \subset E$ be a compact absolutely convex neighborhood of $0$.
    Let $K := p + K_0$.
    Let $f: K \to F$ be an analytic map.
    Let $m, n \in \mathbb{Z}_{\ge 0}$ with $m \le n$.
    Let $\mathbf{C}_m$ be the representation matrix of $\widetilde{f_*}|_{\widetilde{\mathfrak{D}}_{p,m}}: \widetilde{\mathfrak{D}}_{p,m} \to \widetilde{\mathfrak{D}}_{f(p),m}$ with respect to the bases $(u_{p,i})_{i=1}^{r_m^E}$ and $(v_{f(p),i})_{i=1}^{r_m^F}$.
    Let $\mu$ be a probability measure on $K_0$ such that ${\rm supp}(\mu) \subset K_0$.
    Let $z_0^1,\dots, z_0^N \in {\rm supp}(\mu)$ and define $\widehat{\mu}_N := N^{-1} \sum_{i=1}^N\delta_{z_0^i}$.
    Let $Z_N := (p + z_0^1,\dots,p + z_0^N) \in K^N$.
    Assume that $\mathbf{D}_n(\mu)$ is invertible and that there exists $\gamma \in (0,1)$ such that
    \begin{align}
        \left\|\mathbf{I} - \mathbf{D}_n(\mu)^{-1/2}\mathbf{D}_n(\widehat{\mu}_N)\mathbf{D}_n(\mu)^{-1/2}\right\|_{\rm op} \le 1-\gamma. \label{sample number}
    \end{align}
    Then, we have
    \begin{align}
        \left\| \mathbf{C}_m - \widehat{\mathbf{C}}_{m,n,Z_N}\right\|_{\rm Fr} \le C 
        \sqrt{\frac{m!}{\gamma}}
        \cdot \frac{R_\mu^n}{\Lambda_n(\mu)^{1/2}},
        \label{target inequality in main thm}
    \end{align}
    where $R_\mu := \sup_{x \in {\rm supp}(\mu)}\NormCpt{x}{K_0}$ and $C > 0$ is a constant only depending on $f$, $p$, $K_0$, and $\mu$. 
\end{theorem}
\begin{proof}
    By Lemma \ref{lem: main lemma}, we have
    \begin{align*}
        \left\| \mathbf{C}_m - \widehat{\mathbf{C}}_{m,n,Z_N}\right\|_{\rm Fr}
        &\le C 
        \frac{\sqrt{m!}}{\Lambda_n(\widehat{\mu}_N)^{1/2}}
        \sqrt{\int \NormCpt{z}{K_0}^{2n} {\rm d}\widehat{\mu}_N(z)}  \\
        &\le C 
        \frac{\sqrt{m!}}{\Lambda_n(\widehat{\mu}_N)^{1/2}}
        R_\mu^n .
    \end{align*}
    Moreover, by condition \eqref{sample number} and Lemma \ref{lem: perturbation of matrix inverse}, we have
    \[
        \Lambda_n(\widehat{\mu}_N)^{-1}
        \le
        \gamma^{-1}\Lambda_n(\mu)^{-1}.
    \]
    Therefore, we have \eqref{target inequality in main thm}.
\end{proof}

\begin{remark}
    Let $F(\mu, \widehat{\mu}_N) := \|\mathbf{I} - \mathbf{D}_n(\mu)^{-1/2}\mathbf{D}_n(\widehat{\mu}_N)\mathbf{D}_n(\mu)^{-1/2}\|_{\rm Fr}$.
    Since $\|\cdot \|_{\rm op} \le \|\cdot \|_{\rm Fr}$, we easily see that $F(\mu, \widehat{\mu}_N) \le 1 - \gamma$ is a sufficient condition for \eqref{sample number}. 
    As in the following discussion, we have $F(A_*\mu, A_*\widehat{\mu}_N) = F(\mu, \widehat{\mu}_N)$ for any invertible linear map $A: E \to E$, in other words, we may replace $\mu$ and $\widehat{\mu}_N$ with the ones transformed under the pushforward of a linear map when we check the sufficient condition $F(\mu, \widehat{\mu}_N) \le 1 - \gamma$.
    In fact, let $\nu := \widehat{\mu}_N$.
    Then by direct calculation, 
    \begin{align*}
        F(\mu, \nu)^2 
        & = \| \mathbf{D}_n(\mu)^{-1/2}( \mathbf{D}_n(\mu) -\mathbf{D}_n(\nu)) \mathbf{D}_n(\mu)^{-1/2} \|_{\rm Fr}^2 \\
        & = {\rm tr}\left(  \mathbf{D}_n(\mu)^{-1} \left( \mathbf{D}_n(\mu) - \mathbf{D}_n(\nu) \right) \mathbf{D}_n(\mu)^{-1} \left( \mathbf{D}_n(\mu) -\mathbf{D}_n(\nu)\right) \right) 
    \end{align*}
    For an invertible linear map $A: E \to E$, we denote by $\mathbf{S}_n(A) \in \mathbb{C}^{r_n^E \times r_n^E}$ the transpose of the symmetric product of $A$, which is the representation matrix of the pull-back map $A^*: \mathcal{P}_n(E) \to \mathcal{P}_n(E)$ with respect to the monomial basis.
    Then, we see that $\mathbf{D}_n(A_*\mu) = \mathbf{S}_n(A)^* \mathbf{D}_n(\mu) \mathbf{S}_n(A)$ for any measure $\mu$.
    Therefore, we have $F(A_*\mu, A_*\nu) = F(\mu, \nu)$.
\end{remark}

\begin{proof}[Proof of Theorem \ref{thm: main thm, intro}]
    Let $\widehat{\mu}_N := N^{-1}\sum_{i=1}^N\delta_{x^i}$.
    Since $\mu$ is absolutely continuous with respect to the Lebesgue measure on $K_0 \cap \mathbb{R}^d$, no non-zero polynomial of degree at most $n$ vanishes $\mu$-almost everywhere. Hence $\mathbf{D}_n(\mu)$ is positive definite.

    By the weak convergence of $N^{-1}\sum_{i=1}^N \delta_{x^i}$ to $\mu$, each moment of degree at most $2n$ converges, and therefore $\mathbf{D}_n(\widehat{\mu}_N)\to \mathbf{D}_n(\mu)$ as $N \to \infty$.
    In particular, for all sufficiently large $N$, the matrix $\mathbf{D}_n(\widehat{\mu}_N)$ is invertible, and its smallest eigenvalue converges to $\Lambda_n(\mu)$.

    Applying Lemma \ref{lem: main lemma} to $\widehat{\mu}_N := N^{-1}\sum_{i=1}^N\delta_{x^i}$ and taking $\limsup_{N\to\infty}$, we obtain
    \[
        \limsup_{N \to \infty}\left\| \mathbf{C}_m - \widehat{\mathbf{C}}_{m,n,X_N}\right\|_{\rm Fr}
        \lesssim
        \sqrt{m!}\frac{R_\mu^n}{\Lambda_n(\mu)^{1/2}}.
    \]
    This proves the claim.
\end{proof}

When $z^1,\dots, z^N$ are i.i.d.~random variables with a probability measure whose support has non-empty interior, the condition \eqref{sample number} in Theorem \ref{thm: main thm} holds for sufficiently large $N$ as follows:
\begin{lemma}\label{lem: empirical error of G inverse} 
    Let $p \in E$.
    Let $z_0^1, \dots, z_0^N$ be an i.i.d. sequence of $K_0$-valued random variables with a Borel probability measure $\mu$ on $K_0$ and define $\widehat{\mu}_N := N^{-1} \sum_{i=1}^N\delta_{z_0^i}$.
    We assume that $\mathbf{D}_n(\mu)$ is invertible.
    Fix $\delta, \gamma \in (0,1)$, $n \in \mathbb{Z}_{\ge0}$, and $N \in \mathbb{Z}_{> 0}$ such that
    \begin{align}
        N \ge \frac{L_\mu^{4n} \cdot (r_n^E)^2 }{\Lambda_n(\mu)^2}\cdot \frac{\log(2/\delta)}{(1-\gamma)^2}, \label{lower bound of N}
    \end{align}
    where  $L_\mu :=  \max\big(1, \sup_{z \in {\rm supp}(\mu)} \max_{i=1,\dots,d}|z_i|\big)$.
    Then, we have
    \begin{align}
    {\rm Prob}\left[\left\|\mathbf{I} - \mathbf{D}_n(\mu)^{-1/2} \mathbf{D}_n(\widehat{\mu}_N)\mathbf{D}_n(\mu)^{-1/2}\right\|_{\rm op} \le 1-\gamma\right] \ge 1-\delta. \label{estimation of G inverse}
    \end{align}
\end{lemma}
\begin{proof}
    We define an inner product $ \langle \cdot, \cdot \rangle_0$ on $\mathbb{C}^{r_n^{E} \times r_n^{E}}$ by
    \begin{align*}
       \langle \mathbf{A}, \mathbf{B} \rangle_0 := {\rm tr}\left( \mathbf{D}_n(\mu)^{-1/2} \mathbf{A}^*\mathbf{D}_n(\mu)^{-1} \mathbf{B}\mathbf{D}_n(\mu)^{-1/2} \right).
    \end{align*}
    We also define $\| \mathbf{A} \|_0 := \langle \mathbf{A}, \mathbf{A} \rangle_0^{1/2}$.
    Let $\mathbf{w}_n(z) := \big( z^{\alpha^{(i)}} \big)_{i=1}^{r^E_n}$ and regard it as a horizontal vector.
    Then, we have
    \begin{align*}
        \|\mathbf{w}_n(z^i)^*\mathbf{w}_n(z^i) \|_0 
        & = \mathbf{w}_n(z^i) \mathbf{D}_n(\mu)^{-1} \mathbf{w}_n(z^i)^* \\
        & = \mathbf{w}_n(z^i/L_\mu) \mathbf{D}_n((L_\mu^{-1}\mathbf{I})_*\mu)^{-1} \mathbf{w}_n(z^i/L_\mu)^* \\
        & \le 
        \Lambda_n(\mu)^{-1} L_\mu^{2n} r_n^E
   \end{align*}
    By the Hoeffding inequality \cite[Theorem 3.5]{pinelisOptimumBoundsDistributions1994}, we have 
    \begin{align*}
        {\rm Prob}\left[ \left\| \mathbf{D}_n(\mu) - \mathbf{D}_n(\widehat{\mu}_N) \right\|_0 \le 1 - \gamma \right] 
        & \ge 1 - 2\exp\left( - \frac{N\Lambda_n(\mu)^2(1-\gamma)^2}{L_\mu^{4n} \cdot (r_n^E)^2} \right). \label{prob Gn}
    \end{align*}
    Since $\left\|\mathbf{D}_n(\mu)^{-1/2}( \cdot )\mathbf{D}_n(\mu)^{-1/2}\right\|_{\rm op} \le \| \cdot \|_0$, we have \eqref{estimation of G inverse} for $N$ satisfying \eqref{lower bound of N}.
\end{proof} 

As a result, we obtain a more explicit statement when we have i.i.d.~samples as follows:
\begin{corollary}\label{cor: main thm, iid rand. var.}
We use the same notation as in Theorem \ref{thm: main thm}.
Let $\delta \in (0,1)$.
Assume that $z_0^1, \dots, z_0^N$ are i.i.d.~random variables with the probability measure $\mu$.
Then, for any $n \ge m$ and $N \in \mathbb{Z}_{>0}$ satisfying 
\begin{align}
    N &\ge \frac{L_\mu^{4n} \cdot (r_n^E)^2 }{\Lambda_n(\mu)^2}\cdot 4\log(2/\delta),
\end{align}
with probability at least $1 - \delta$, we have
\begin{align}
    \left\| \mathbf{C}_m - \widehat{\mathbf{C}}_{m,n,Z_N}\right\|_{\rm Fr} \le  C 
        \sqrt{2 m!} \cdot \frac{R_\mu^n}{\Lambda_n(\mu)^{1/2}},
\end{align}
\end{corollary}
\begin{proof}
    It follows from Theorem \ref{thm: main thm} and Lemma \ref{lem: empirical error of G inverse}.
\end{proof}

\section{Approximation of analytic maps}\label{sec: approximation of analytic map}
We use the same notation as in Section \ref{sec: approximations of push-forwards}.
For each $z \in E$ and $m, n \in \mathbb{Z}_{\ge 0}$ with $m \le n$, we define $\widehat{f}_{m,n,Z}(z) \in F$ by 
\begin{align}
    \widehat{f}_{m,n,Z,i}(z) &:= \mathbf{u}_{p,m}(z) \widehat{\mathbf{C}}_{m,n,Z}^*(\partial_{w_i}\mathbf{v}_{f(p),m}(0))^*,
\end{align}
for $i=1,\dots,r$.
Then, we have the following theorem:
\begin{theorem}\label{thm: regression}
    Let $p \in E \cap \mathbb{R}^d$ and let $K_0 \subset E$ be a compact absolutely convex neighborhood of $0$.
    Let $K := p + K_0$.
    Let $f: K \to F$ be an analytic map.
    Let $m, n \in \mathbb{Z}_{\ge 0}$ with $m \le n$.
    Let $\mu$ be a Borel probability measure on $K_0$ such that ${\rm supp}(\mu) \subset K_0$.
    Let $z_0^1,\dots, z_0^N \in {\rm supp}(\mu)$ and define $\widehat{\mu}_N := N^{-1} \sum_{i=1}^N\delta_{z_0^i}$.
    Let $Z_N := (p+z_0^1,\dots,p+z_0^N) \in K^N$.
    Assume that $\mathbf{D}_n(\mu)$ is invertible and that there exists $\gamma \in (0,1)$ such that
    \begin{align}
        \left\|\mathbf{I} - \mathbf{D}_n(\mu)^{-1/2}\mathbf{D}_n(\widehat{\mu}_N)\mathbf{D}_n(\mu)^{-1/2}\right\|_{\rm op} \le 1-\gamma. 
    \end{align}
    Then, for any compact subset $K_1 \Subset K_0$, we have
\begin{align}
    &\sup_{z \in p + K_1} \big\|\widehat{f}_{m,n,Z_N}(z) - f(z) \big\|_F
    \le
    C'\left(
    \sqrt{\frac{m!}{\gamma}}
    \cdot \frac{R_\mu^n}{\Lambda_n(\mu)^{1/2}}
    + R_{K_1}^m
    \right)
    \label{reconstruction error}
\end{align}
where $R_\mu := \sup_{x \in {\rm supp}(\mu)}\NormCpt{x}{K_0}$, $R_{K_1} := \sup_{z \in K_1}\NormCpt{z}{K_0}$ and $C'$ is a constant only depending on $f$, $p$, $K_0$, $K_1$, and $\mu$.
\end{theorem}
\begin{proof}
    For $w \in F$, we have $\langle f(z), w\rangle_F = w^*(f(z)) = \langle w^*, \mathbf{e}_{f(z)} \rangle_{H_F}$, where we use Proposition \ref{prop: RKHS} in the last identity.
    Since $\widetilde{f_*}(\mathbf{e}_z) = \mathbf{e}_{f(z)}$, we have
    \begin{align}
        \langle f(z), w \rangle_F = \langle w^*, \widetilde{f_*}(\mathbf{e}_z) \rangle_{H_F} \label{f with inner product}.
    \end{align}
    Since $\widetilde{f_*}(\pi_{p, m}\mathbf{e}_z) \in \widetilde{\mathfrak{D}}_{f(p),m}$ and
    \begin{align*}
  \pi_{f(p), m} w^* = \lim_{|a| \to 0} \pi_{f(p), m} a^{-1} (\mathbf{e}_{\overline{a}w} - 1) = \mathbf{v}_{f(p),m} \cdot \left( \nabla_{w} \mathbf{v}_{f(p),m}(0)\right)^*
    \end{align*} 
     hold, we see that
    \begin{align}
        \langle w^*, \widetilde f_*(\pi_{p, m}\mathbf{e}_z) \rangle_{H_F}
        = \langle \pi_{f(p),m} w^*, \widetilde f_*(\pi_{p, m}\mathbf{e}_z) \rangle_{H_F} 
        = \mathbf{u}_{p,m}(z) \mathbf{C}_m^* \left(\nabla_{w}\mathbf{v}_{f(p),m}(0)\right)^*.
        \label{f pi with inner product}
    \end{align}
    Thus, using \eqref{f with inner product} and \eqref{f pi with inner product}, for $z \in p + K_1$, there exists a compact absolutely convex set $K_2 \Supset K_0$ such that for any $w \in F$ with $\|w\|_F=1$, we have
    \begin{align}
        &\left|\big\langle \widehat{f}_{m,n,Z_N}(z) - f(z), w \big\rangle_F\right| \notag \\
        &\le \left|\mathbf{u}_{p,m}(z) (\widehat{\mathbf{C}}_{m,n,Z_N}^* - \mathbf{C}_m^*) \nabla_{w}\mathbf{v}_{f(p), m}(0)^*\right|
        + \|w^*\|_{H_F} \cdot \left\| \widetilde{f_*}(\pi_{p,m}\mathbf{e}_z - \mathbf{e}_z)\right\|_{H_F} \notag \\
        & \lesssim 
        e^{\sup_{z\in p+K_1}\|z\|_E^2/2} \big\|\widehat{\mathbf{C}}_{m,n,Z_N} - \mathbf{C}_m \big\|_{\rm op}
        + \big\|\pi_{p,m}\mathbf{e}_z - \mathbf{e}_z \big\|_{\exp, p + K_2}, \label{last inequality}
    \end{align}
where we use  Proposition \ref{prop: gelfand triple} and the continuity of $\widetilde{f_*}$.
Therefore, by Lemmas \ref{lem: proj in Exp} and Theorem \ref{thm: main thm}, we have \eqref{reconstruction error}.
\end{proof}
\begin{proof}[Proof of Theorem \ref{thm: main thm 2, intro}]
    Fix $m,n$ with $m\le n$ and let $\widehat{\mu}_N := \frac{1}{N}\sum_{i=1}^N \delta_{x^i}$.
    Since $\mu$ is absolutely continuous with respect to the Lebesgue measure on $K_0\cap\mathbb{R}^d$, the matrix $\mathbf{D}_n(\mu)$ is positive definite. Moreover, by the weak convergence of $\widehat{\mu}_N$ to $\mu$, we have $\mathbf{D}_n(\widehat{\mu}_N)\to \mathbf{D}_n(\mu)$ as $N\to\infty$. 
    Hence, for all sufficiently large $N$, the condition in Theorem \ref{thm: regression} holds with, for example, $\gamma=1/2$.

    Applying Theorem \ref{thm: regression} and taking $\limsup_{N\to\infty}$ gives
    \[
        \limsup_{N \to \infty}\sup_{z \in p + K_1}
        \big\|\widehat{f}_{m,n,X_N}(z)-f(z)\big\|_F
        \le
        C\left(
        \sqrt{m!}\frac{R_\mu^n}{\Lambda_n(\mu)^{1/2}}
        + R_{K_1}^m
        \right),
    \]
    after absorbing the factor $\sqrt{2}$ into the constant.

    For the i.i.d. statement, Lemma \ref{lem: empirical error of G inverse} with $\gamma=1/2$ implies that the hypothesis of Theorem \ref{thm: regression} holds with probability at least $1-\delta$ under the stated lower bound on $N$. Applying Theorem \ref{thm: regression} again gives the desired probabilistic estimate.
\end{proof}

Using the above theorems, we have a result for the least-squares polynomial approximation as follows:
\begin{theorem}\label{thm: convergence of least-squares polynomial}
    Let $\mu$ be a Borel probability measure on $K_0 \cap \mathbb{R}^d$ absolutely continuous with respect to the Lebesgue measure on $K_0 \cap \mathbb{R}^d$.
    Let $g \in \mathcal{O}_E(K_0)$.
    Let $\{x^i\}_{i=1}^\infty \subset K_0 \cap \mathbb{R}^d$.
    Let $P_{n,N}(x)=\sum_{|\alpha|\le n}c_\alpha x^\alpha$ be the least-squares polynomial of degree $n$ for the data $(x^1,g(x^1)),\dots,(x^N,g(x^N))$ whose coefficient vector is chosen by the Moore--Penrose pseudo-inverse, i.e. the minimum-norm least-squares solution of
    \begin{align*}
        \min \left\{ \sum_{i=1}^N |P(x^i) - g(x^i)|^2 
        : P\text{: polynomial},~{\rm deg}(P) \le n \right\}.
    \end{align*}
    For $m \le n$, we define the truncated polynomial $\widehat{P}_{m,n,N}(x) := \sum_{|\alpha| \le m} c_{\alpha} x^\alpha$ of degree $m$.
    Assume that  $N^{-1}\sum_{i=1}^N \delta_{x^i}$ weakly converges to $\mu$ as $N \to \infty$.
    Let $K_1 \Subset K_0$ be a compact subset.
    Then, for any $m, n \in \mathbb{Z}_{> 0}$ with $m \le n$, we have
    \begin{align}
        \limsup_{N \to \infty}\sup_{z \in K_1} \big|\widehat{P}_{m,n,N}(z) - g(z) \big|
        \le 
        C' \left(
        \sqrt{m!}
        \cdot \frac{R_\mu^n}{\Lambda_n(\mu)^{1/2}}
        + R_{K_1}^m
        \right),
        \label{main thm 2 intro}
    \end{align}
    where $R_\mu := \sup_{x \in {\rm supp}(\mu)}\NormCpt{x}{K_0}$, $R_{K_1} := \sup_{z \in K_1}\NormCpt{z}{K_0}$,  and $C'$ is a constant only depending on $g$, $K_0$, $K_1$, and $\mu$.
    Furthermore, let $L_\mu :=  \max\big(1, \sup_{z \in {\rm supp}(\mu)} \max_{i=1,\dots,d}|z_i|\big)$.
    Assume that $x^1, \dots, x^N$ are i.i.d. random variables with the probability measure $\mu$.
    Then, for $\delta \in (0,1)$ and $m,n \in \mathbb{Z}_{> 0}$ and $N \in \mathbb{Z}_{>0}$ with $m \le n$ such that  
    \begin{align*}
        N &\ge \frac{L_\mu^{4n} \cdot (r_n^E)^2 }{\Lambda_n(\mu)^2}\cdot 4\log(2/\delta), 
    \end{align*}
    we have
    \begin{align}
    {\rm Prob}\left[\sup_{z \in K_1} \big|\widehat{P}_{m,n,N}(z) - g(z) \big|
    \le C' \left(
        \sqrt{2m!}
        \cdot \frac{R_\mu^n}{\Lambda_n(\mu)^{1/2}}
        + R_{K_1}^m
        \right) \right] \ge 1-\delta.
    \label{reconstruction error, polynomial regression}
    \end{align}
\end{theorem}
\begin{proof}
    Let $P_n(x) := \sum_{|\alpha| \le n} c_{\alpha} x^\alpha$. 
    Let $E := \mathbb{C}^d$ and $F := \mathbb{C}$.
    We define the analytic map $f: K_0 \to F$ by $f(z) := g(z) - g(0)$.
    Let $X := (x^1,\dots, x^N)$ and $Y := (f(x^1),\dots, f(x^N))$.

    Since $\mu$ is absolutely continuous with respect to the Lebesgue measure and is a probability measure, no non-zero polynomial can vanish $\mu$-almost everywhere. 
    Hence $\mathbf{D}_n(\mu)$ is positive definite.
    Moreover, by the weak convergence of $N^{-1}\sum_{i=1}^N\delta_{x^i}$ to $\mu$, we have $\mathbf{D}_n(\widehat{\mu}_N)\to \mathbf{D}_n(\mu)$.
    Thus, $\mathbf{U}_{0,n,X}$ is full column rank for all sufficiently large $N$.
    Therefore, in the following limsup argument, we may assume that $\mathbf{U}_{0,n,X}$ is full column rank.

    Let $\mathbf{p}_n(x) := (1,\cdots, x^{\alpha^{(i)}}, \cdots, x^{\alpha^{(r_n^{E})}})$ be a horizontal vector and let $\mathbf{F}_n = {\rm diag}\left( (\sqrt{\alpha^{(i)}!})_{i=1}^{r_n^E}\right)$.
    Since $\mathbf{U}_{0,n,X}$ is full column rank, by the identity $\mathbf{u}_{0,n}(x) = \mathbf{p}_n(x) \mathbf{F}_n^{-1}$,  we see that
    \[
    \begin{pmatrix}
                \mathbf{p}_n(x^1)\\
                \vdots \\
                \mathbf{p}_n(x^N)
        \end{pmatrix}^{\dagger}
        = \mathbf{F}_n^{-1} \mathbf{U}_{0,n, X}^\dagger.
    \]
    Thus, we have
    \begin{align*}
    \begin{pmatrix}
            \mathbf{p}_n(x^1)\\
            \vdots \\
            \mathbf{p}_n(x^N)
    \end{pmatrix}^{\dagger}
    \begin{pmatrix}
            g(0)\\
            \vdots \\
            g(0)
    \end{pmatrix}
    = \begin{pmatrix}
            g(0)\\
            0\\
            \vdots \\
            0
        \end{pmatrix},~~~
    \end{align*}
    and
    \begin{align*}
        P_n(x) &=  \mathbf{p}_n(x)
        \begin{pmatrix}
                \mathbf{p}_n(x^1)\\
                \vdots \\
                \mathbf{p}_n(x^N)
        \end{pmatrix}^{\dagger}
        \begin{pmatrix}
                f(x^1)\\
                \vdots \\
                f(x^N)
        \end{pmatrix}
     + g(0)
     \\
        &= \mathbf{u}_{0,n}(x)
        (\mathbf{U}_{0,n,X})^\dagger  \mathbf{V}_{0,n,Y} 
        (\partial_{w_1}\mathbf{v}_{0,n})(0)^* + g(0).
    \end{align*}
    Thus, the truncated polynomial $\widehat{P}_{m,n,N}$ coincides with $\widehat{f}_{m,n,X} + g(0)$, where $\widehat{f}_{m,n,X}: E \to F$ is the map defined in Theorem \ref{thm: regression}.
    Then, the inequalities \eqref{main thm 2 intro} and \eqref{reconstruction error, polynomial regression} follow from Theorem \ref{thm: regression} and Lemma \ref{lem: empirical error of G inverse} with $\gamma = 1/2$.
\end{proof}
Theorem \ref{thm: convergence of least-squares polynomial} implies that if we have a sufficiently large number $N$ of data points, the truncated polynomial $\widehat{P}_{m,n,N}$ constructed from the least-squares polynomial can uniformly approximate $g$ on any compact subset of $K_0 \cap \mathbb{R}^d$ for suitable $m$ and $n$ with $m \le n$, if we assume that $K_0$ is sufficiently large compared to ${\rm supp}(\mu)$.

We end this section with several lemmas that we use in the proofs of the theorems.
\begin{lemma} \label{lem: perturbation of matrix inverse}
Let $0 < \gamma < 1$.
Let $\mathbf{A}$ and $\mathbf{B}$ be positive definite matrices such that $\|\mathbf{A}^{-1/2}(\mathbf{A} - \mathbf{B})\mathbf{A}^{-1/2}\|_{\rm op} \le 1-\gamma$. 
Then, we have
 \begin{align}
     \left\|\mathbf{I} - \mathbf{A}^{1/2}\mathbf{B}^{-1}\mathbf{A}^{1/2} \right\|_{\rm op} 
     \le \gamma^{-1} - 1.
 \end{align}
\end{lemma}
\begin{proof}
Let $\mathbf{E} := \mathbf{A}^{-1/2}(\mathbf{A} - \mathbf{B})\mathbf{A}^{-1/2}$.
Since $\mathbf{A}^{1/2}\mathbf{B}^{-1}\mathbf{A}^{1/2} - \mathbf{I} = (\mathbf{I} - \mathbf{E})^{-1} - \mathbf{I}$, it follows from the following estimate:
\begin{align*}
    \|\mathbf{A}^{1/2}\mathbf{B}^{-1}\mathbf{A}^{1/2} - \mathbf{I} \|_{\rm op}
    \le \sum_{i=1}^\infty \|\mathbf{E}\|_{\rm op}^i \le \frac{1-\gamma}{1 - (1-\gamma)} = \gamma^{-1} - 1.
\end{align*}
\end{proof}

\section{Approximation of analytic vector fields}\label{sec: reconstruction of vector fields}

We use the same notation as in Section \ref{sec: approximations of push-forwards}.
Let $K \subset E$ be a compact convex set.
Let $V: K \to E$ be an analytic map.
We consider the ordinary differential equation \eqref{ODE}.
Suppose that there exists $T>0$ such that for any $z^0 \in K$, the ordinary differential equation \eqref{ODE} has a solution on $[0,T]$.
For $t \in [0,T]$, we define the flow map $\phi^t: K \to E$ by $\phi^t(w) := z(t)$ where $z(t)$ is the solution of \eqref{ODE} with initial point $z^0 = w$.
We define a continuous linear map $\nabla_V: \mathcal{O}_E(K) \to \mathcal{O}_E(K)$ by
\begin{align}
    (\nabla_Vh)(z) := (\nabla_{V(z)} h)(z).
\end{align}
We remark that $\nabla_Vh = \lim_{t \to 0} t^{-1}((\phi^t)^*h - h)$.
Then, we have the following proposition:
\begin{proposition}
    Let $p \in K$.
    Then, for each $n \ge 0$, we have 
    \begin{align}
    \left(\nabla_V' - \sum_{i=1}^d V_i(p) \partial_{z_i}'\right)(\mathfrak{D}_{p,n}) \subset \mathfrak{D}_{p,n}
    \end{align}
    In particular, we have $\nabla_V'(\mathfrak{D}_{p,n}) \subset \mathfrak{D}_{p,n}$ if $V(p)=0_E$.
\end{proposition}
\begin{proof}
    It suffices to show that
    \begin{align}
        \left(\nabla_V' - \sum_{i=1}^d V_i(p) \partial_{z_i}'\right) \delta_p^{(\alpha)}\in \mathfrak{D}_{p,n}
        \label{eq:invariance_of_Dpn}
    \end{align}
    for any $\alpha$ with $|\alpha| \le n$.
    By definition $\nabla_V' \delta_p^{(\alpha)} (h) = \delta_p^{(\alpha)} \left(\sum_{i=1}^d V_i\partial_{z_i}h\right)$ holds.
    Thus, we have 
    \[\nabla_V' \delta_p^{(\alpha)} (h) = \sum_{i=1}^d V_i(p) \partial_{z_i}'(\delta_p^{(\alpha)})(h) + \mu(h)\]
     for some $\mu \in \mathfrak{D}_{p,n}$, namely, \eqref{eq:invariance_of_Dpn} holds.
\end{proof}

For $t \in [0,T]$, let $\mathbf{C}^{(t)}_n \in \mathbb{C}^{r^E_n \times r_n^E}$ be the representation matrix of $\widetilde{\phi_*^t}|_{\widetilde{\mathfrak{D}}_{p,n}}: \widetilde{\mathfrak{D}}_{p,n} \to \widetilde{\mathfrak{D}}_{\phi^t(p),n}$ with respect to the basis $\{ u_{p,i} \}_{i=1}^{r_n^E}$ and $\{ u_{\phi^t(p), i} \}_{i=1}^{r_n^E}$.
Then, we have the following proposition:
\begin{proposition}
    Assume $V(p) = 0_E$.
    Let $\mathbf{A}_n$ be the representation matrix of $\widetilde{\nabla_V'}|_{\widetilde{\mathfrak{D}}_{p,n}}: \widetilde{\mathfrak{D}}_{p,n} \to \widetilde{\mathfrak{D}}_{p,n}$ with respect to the basis $\{ u_{p,i} \}_{i=1}^{r_n^E}$.
    Then, for $t \in [0, T]$, we have
    \begin{align}
        e^{t\mathbf{A}_n} = \mathbf{C}^{(t)}_n. \label{expA = C}
    \end{align}
\end{proposition}
\begin{proof}
    Since $V(p) = 0_E$, we have $\phi^t(p) = p$ for any $t \in [0, T]$.
    For $s, t \in [0,T]$ with $s+t \in [0,T]$, we have $\mathbf{C}^{(t+s)}_n = \mathbf{C}^{(t)}_n \mathbf{C}^{(s)}_n$.
    Since $\nabla_Vh = \lim_{t \to 0} t^{-1}((\phi^t)^*h - h)$, we have
    \begin{align*}
        \frac{{\rm d}}{{\rm d}t}\mathbf{C}^{(t)}_n = \mathbf{C}^{(t)}_n \mathbf{A}_n.
    \end{align*}
    Thus,  we have \eqref{expA = C}.
\end{proof}

Assume that $V(p) = 0_E$.
Let $f= \phi^T$ and let $\widehat{\mathbf{C}}_{m,n,Z} \in \mathbb{C}^{r_m^E \times r_m^E}$ be the matrix defined in \eqref{C_mnN} with respect to the point $p\in E$, $Z = (z^1,\dots, z^N) \in K^N$, and $f = \phi^T$.
When $\widehat{\mathbf{C}}_{m,n,Z}$ has no non-positive real eigenvalues, for $m, n \in \mathbb{Z}_{\ge 0}$ with $m \le n$, we define 
\begin{align}
    \widehat{\mathbf{A}}_{m,n,Z} := \frac{1}{T}\log \widehat{\mathbf{C}}_{m,n,Z},
\end{align}
where $\log$ is the matrix logarithm, which is defined by 
\begin{align}
    \log(\mathbf{A}) := (\mathbf{A} - \mathbf{I})\int_0^1 (\mathbf{I} + t(\mathbf{A} -\mathbf{I}))^{-1} {\rm d}t
    \label{matrix log}
\end{align}
for a square matrix $\mathbf{A}$ with no non-positive real eigenvalues \cite{highamFunctionsMatrices2008}.
For each $z \in E$ and $m, n \in \mathbb{Z}_{\ge 0}$ with $m \le n$, we define $\widehat{V}_{m,n,Z}(z) \in E$ by
\begin{align}
    \widehat{V}_{m,n,Z,i}(z) &:= \mathbf{u}_{p,m}(z) \widehat{\mathbf{A}}_{m,n,Z}^*(\partial_{z_i}\mathbf{u}_{p,m}(0))^*
\end{align}
for $i=1,\dots,d$.
For each $n \ge 0$, we define
\begin{align}
    B_n := \sup_{t \in [0,1]} \left\| \left(\mathbf{I}_{r_n^E} + t\big(\mathbf{C}_n^{(T)} - \mathbf{I}_{r_n^E}\big)\right)^{-1}\right\|_{\rm op}.
\end{align}
Then, we prove the following theorem:
\begin{theorem}\label{thm: reconstruction of vector field}
    Let $p \in E \cap \mathbb{R}^d$ and let $K_0 \subset E$ be a compact absolutely convex neighborhood of $0$.
    Let $K := p + K_0$.
    Let $V: K \to E$ be an analytic map such that $V(p) = 0_E$.
    Let $m, n \in \mathbb{Z}_{\ge 0}$ with $m \le n$, and $N \in \mathbb{Z}_{>0}$.
    Let $\mu$ be a Borel probability measure on $K_0 \cap \mathbb{R}^d$ such that ${\rm supp}(\mu) \subset K_0 \cap \mathbb{R}^d$.
    Let $z_0^1,\dots, z_0^N \in {\rm supp}(\mu)$ and define $\widehat{\mu}_N := N^{-1} \sum_{i=1}^N\delta_{z_0^i}$.
    Let $Z_N := (p+z_0^1,\dots,p+z_0^N) \in K^N$.
    Assume that
\begin{enumerate}[label = (\roman*)]
    \item the eigenvalues of the Jacobian matrix ${\rm d}V_p$ are real,
    \item $\mathbf{D}_n(\mu)$ is invertible and there exists $\gamma_1 \in (0,1)$ such that  
    \begin{align*}
        \left\|\mathbf{I} - \mathbf{D}_n(\mu)^{-1/2}\mathbf{D}_n(\widehat{\mu}_N)\mathbf{D}_n(\mu)^{-1/2}\right\|_{\rm op} \le 1-\gamma_1
    \end{align*}
    \item there exists $\gamma_2 \in (0,1)$ such that $\left\| \mathbf{C}^{(T)}_m - \widehat{\mathbf{C}}_{m,n,Z_N} \right\|_{\rm op} < (1 - \gamma_2) B_m^{-1}$. \label{CT - CmnZ}
\end{enumerate}
Then, for any compact subset $K_1 \Subset K_0$, we have
\begin{align}
    \sup_{z \in p + K_1} \big\|\widehat{V}_{m,n,Z_N}(z) - V(z) \big\|_E \le C'\left(
    \frac{B_m^2}{\gamma_2}\sqrt{\frac{m!}{\gamma_1}}
    \frac{R_\mu^n}{\Lambda_n(\mu)^{1/2}}
    + R_{K_1}^m
    \right)
    \label{reconstruction error, vector field}
\end{align}
where $R_\mu := \sup_{x \in {\rm supp}(\mu)}\NormCpt{x}{K_0}$, $R_{K_1} := \sup_{z \in K_1}\NormCpt{z}{K_0}$, and $C'$ is a constant only depending on $V$, $p$, $K_0$, $K_1$, and $\mu$.
\end{theorem}
\begin{proof}
    Let $\mathbf{C}_m := \mathbf{C}_m^{(T)}$.
    Since the eigenvalues of $\mathbf{A}_m$ are real, we have $\log \mathbf{C}_m = T\mathbf{A}_m$.
    Let $\mathbf{D}_{m,t} := \mathbf{I} + t(\mathbf{C}_m - \mathbf{I})$.
    Then, for $t \in [0,1]$, we have
    \begin{align}
        \mathbf{I} + t(\widehat{\mathbf{C}}_{m,n,Z_N} - \mathbf{I}) 
        = (\mathbf{I} +  t(\widehat{\mathbf{C}}_{m,n,Z_N} - \mathbf{C}_m )\mathbf{D}_{m,t}^{-1})\mathbf{D}_{m,t}. \label{I - t hatC - I}
    \end{align}
    Since $\|t(\mathbf{C}_m - \widehat{\mathbf{C}}_{m,n,Z_N})\mathbf{D}_{m,t}^{-1}\|_{\rm op} < 1-\gamma_2$ by the condition \ref{CT - CmnZ}, we see that $\widehat{\mathbf{C}}_{m,n,Z_N}$ has no non-positive real eigenvalues and that
    \begin{align}
        \left\|(\mathbf{I} + t(\widehat{\mathbf{C}}_{m,n,Z_N} - \mathbf{I}))^{-1}\right\|_{\rm op} \le \gamma_2^{-1}B_m \label{I + t(C-I)}
    \end{align}
    holds.
    By the same argument as in the proof of Theorem \ref{thm: regression}, there exists a compact absolutely convex neighborhood $K_2 \Supset K_0$ of $0$ such that 
    \begin{align*}
    &\left\|\widehat{V}_{m,n,Z_N}(z) - V(z)\right\|_E \\
    &\lesssim 
    \frac{e^{\sup_{\zeta\in p+K_1}\|\zeta\|_E^2/2}}{T}
    \big\|\log\widehat{\mathbf{C}}_{m,n,Z_N} - \log\mathbf{C}_m \big\|_{\rm op} 
    + \big\|\pi_{p,m}\mathbf{e}_z - \mathbf{e}_z \big\|_{\exp, p+K_2}.
\end{align*}
    Thus, it suffices to estimate $\big\|\log\widehat{\mathbf{C}}_{m,n,Z_N} - \log\mathbf{C}_m \big\|_{\rm op}$.
    By direct calculation, we have
    \begin{align*}
        &\log\widehat{\mathbf{C}}_{m,n,Z_N} - \log\mathbf{C}_m \\
        &= (\widehat{\mathbf{C}}_{m,n,Z_N} - \mathbf{I})\int_0^1 (\mathbf{I} + t(\widehat{\mathbf{C}}_{m,n,Z_N} -\mathbf{I}))^{-1} {\rm d}t 
        - (\mathbf{C}_m - \mathbf{I})\int_0^1 (\mathbf{I}_{r_m^E} + t(\mathbf{C}_m -\mathbf{I}))^{-1} {\rm d}t \\
        & =  \int_0^1(\mathbf{I} + t(\widehat{\mathbf{C}}_{m,n,Z_N} -\mathbf{I}))^{-1} 
         (\widehat{\mathbf{C}}_{m,n,Z_N} - \mathbf{C}_m)  \mathbf{D}_{m,t}^{-1} {\rm d}t\\
         & = 
          \int_0^1 \mathbf{D}_{m,t}^{-1} (\mathbf{I} +  t( \widehat{\mathbf{C}}_{m,n,Z_N} - \mathbf{C}_m)\mathbf{D}_{m,t}^{-1})^{-1} 
         (\widehat{\mathbf{C}}_{m,n,Z_N} - \mathbf{C}_m)  \mathbf{D}_{m,t}^{-1} {\rm d}t.
    \end{align*}
    Therefore, using \eqref{I + t(C-I)} and the Neumann series estimate (see the proof of Lemma \ref{lem: perturbation of matrix inverse}) 
    \[
        \left\|
        \left(\mathbf{I} + t( \widehat{\mathbf{C}}_{m,n,Z_N} - \mathbf{C}_m)\mathbf{D}_{m,t}^{-1}\right)^{-1}
        \right\|_{\rm op}
        \le \gamma_2^{-1},
    \]
    we have
    \begin{align*}
        \big\|\log\widehat{\mathbf{C}}_{m,n,Z_N} - \log\mathbf{C}_m \big\|_{\rm op} 
        \le \frac{B_m^2}{\gamma_2}
        \left\|\widehat{\mathbf{C}}_{m,n,Z_N} - \mathbf{C}_m\right\|_{\rm Fr}.
    \end{align*}
    By the same argument as in the proof of Theorem \ref{thm: regression}, we have \eqref{reconstruction error, vector field}.
\end{proof}
\begin{proof}[Proof of Theorem \ref{thm: main thm 3, intro}]
    Fix $m$ and a compact set $K_1 \Subset K_0$. By the assumption on $\rho$, after replacing the assumption by the equivalent estimate
    \[
        \frac{R_\mu^n}{\Lambda_n(\mu)^{1/2}} \lesssim \rho^n
    \]
    for all sufficiently large $n$, Theorem \ref{thm: main thm} gives
    \[
        \limsup_{N\to\infty}
        \|\mathbf{C}_m^{(T)}-\widehat{\mathbf{C}}_{m,n,X_N}\|_{\rm op}
        \le
        \limsup_{N\to\infty}
        \|\mathbf{C}_m^{(T)}-\widehat{\mathbf{C}}_{m,n,X_N}\|_{\rm Fr}
        \lesssim
        \sqrt{m!}\,\rho^n .
    \]
    Hence, for sufficiently large $n$ and then sufficiently large $N$, condition \ref{CT - CmnZ} in Theorem \ref{thm: reconstruction of vector field} is satisfied with some $\gamma_2\in(0,1)$. Applying Theorem \ref{thm: reconstruction of vector field} and taking $N\to\infty$ yields the desired estimate, after absorbing the constants depending on $m$ into $C_m$.
\end{proof}
\begin{remark}
The assumption that the Jacobian matrix has real eigenvalues is essential for ensuring convergence for all $m$, due to the arbitrariness in choosing branches for the matrix logarithm.
If we take suitably small $m$ and $T$, we can prove a similar convergence result.
\end{remark}

\bibliography{reference}
\bibliographystyle{hplain} 

\appendix

\section{Asymptotics of the lowest eigenvalues of Hankel matrices}
\label{appendix: Asymptotics of the lowest eigenvalues of Hankel matrices}
Here, we review the results of \cite[Section 3.3]{wilfFiniteSectionsClassical1970}.
Let $a \in \mathbb{R}$ and $r>0$.
We fix a measurable map $w: [a-r, a+r] \to \mathbb{R}_{\ge 0}$ such that
\begin{align}
    \int_{-1}^1 \frac{\log(w(a+rt))}{\sqrt{1 - t^2}} {\rm d}t > -\infty.
\end{align}
We define the matrix of size $n+1$ by
\begin{align}
    \mathbf{C}_n(a,r) := \left( \int_{a-r}^{a+r} t^{i+j} w(t) {\rm d}t \right)_{i,j=0,\dots, n}.
\end{align}
We define
\begin{align*}
    \sigma_{a,r} = \sigma := 
    \begin{cases}
        \displaystyle \frac{|a|+1}{r} +  \left( \left( \frac{|a|+1}{r} \right)^2 - 1 \right)^{1/2} & \text{ if } |a| + a^2 -r^2 \ge 0, \\[15pt]
        \displaystyle \left(\frac{1}{r^2 - a^2} + 1  \right)^{1/2} + \left(\frac{1}{r^2 - a^2}\right)^{1/2} &\text{ if } |a| + a^2 -r^2 \le 0.
    \end{cases}
\end{align*}
We also define $A(z)$ by 
\begin{align*}
    \log|A(\rho e^{i \phi})| 
    =
    \frac{-1}{4\pi} \int_{-\pi}^\pi \log \big(w(r \cos t + a)| \sin t| \big) \frac{\rho^2 - 1}{1 - 2 \rho \cos(\phi - t) + \rho^2} {\rm d}t
\end{align*}
Let 
\begin{align*}
    \zeta(z) &:= \frac{z-a}{r} + \left[\left(\frac{z - a}{r} \right)^2 - 1\right]^{1/2} \\
    g(\theta) &:= |\zeta(e^{i\theta})|
\end{align*}
\begin{theorem}[{\cite[Theorem 3.2]{wilfFiniteSectionsClassical1970}}]
    \label{thm: asymptotic formula of smallest eigenvalues}
    Let $\lambda_n(a,r)$ be the smallest eigenvalue of $\mathbf{C}_n(a,r)$.
    Then, we have 
    \begin{align*}
        \lambda_n(a,r) \sim &\sigma^{-2n-2} \\
        &\times 
        \begin{cases}
            \displaystyle\frac{\sqrt{2n}(\sigma^{2} - 1)}{\gamma}  & \text{ if } |a| + a^2 -r^2 > 0, \\[15pt]
            \displaystyle\frac{(2n)^{1/4}(\sigma^{2} - 1)}{\gamma}  & \text{ if } |a| + a^2 -r^2 = 0, \\[15pt]
            \displaystyle\frac{2\sqrt{2n}}{\gamma}\left(\frac{1}{\sigma^2 - 1} + \left( \frac{1}{\sigma^4 - 2\sigma^2 \cos 2 \phi_0 - 1} \right)^{1/2}\right)^{-1}  & \text{ if } |a| + a^2 -r^2 < 0, \\ 
        \end{cases}
    \end{align*}
    $\phi_0$ satisfies $e^{i\phi_0} = \zeta(e^{i \theta_0})/|\zeta(e^{i \theta_0})|$ with $\cos \theta_0 = |a|/(a^2 - r^2)$, and 
    \begin{align}
        \gamma &:=
        \begin{cases}
            \displaystyle \frac{|A(\zeta(-1))|^2 \sigma^{1/2}}{2^{3/2}\pi^{5/2}r \sqrt{|g''(\pi)|}} & \text{ if }|a| + a^2 -r^2 > 0, \\[15pt]
            \displaystyle \frac{3^{1/4} \Gamma(1/4) |A(\zeta(-1))|^2 \sigma^{1/4}}{2^{17/4}\pi^{2}r |g^{(4)}(\pi)|^{1/4}} & \text{ if }|a| + a^2 -r^2 = 0, \\[15pt]
            \displaystyle \frac{|A(\zeta(e^{i\theta_0}))|^2 (2 \sigma)^{1/2}}{2\pi^{3/2}r \sqrt{|g''(\theta_0)|}} & \text{ if }|a| + a^2 -r^2 < 0,
        \end{cases}
    \end{align}
\end{theorem}
\begin{remark}
    In \cite{wilfFiniteSectionsClassical1970}, they use the notation $g^{(iv)}$ for the fourth derivative of $g$ (compare \cite[Section 2.4]{wilfFiniteSectionsClassical1970} and \cite[Theorem 3.1]{widomEigenvaluesCertainHermitian1958a}).
\end{remark}
In the rest of this appendix, unless otherwise stated, $\lambda_n(a,r)$ denotes the smallest eigenvalue of $\mathbf{C}_n(a,r)$ in the special case $w\equiv 1$.
\begin{lemma} \label{lem: lower bound}
    Let $p= (p_i)_{i=1}^d \in \mathbb{R}^d$ and let $\mathbf{r} = (r_i)_{i=1}^d \in \mathbb{R}_{>0}^d$.
    Let $\Delta(p, \mathbf{r}) := \prod_{i=1}^d [p_i - r_i, p_i + r_i]$ be a rectangle.
    Let $\mu_{\Delta(p, \mathbf{r})}$ be the Lebesgue measure on $\Delta(p, \mathbf{r})$.
    Then, 
    \begin{align}
        \Lambda_n(\mu_{\Delta(p, \mathbf{r})}) \ge \prod_{i=1}^d \lambda_{n}(p_i,r_i).
        \label{ineq for prod of eigenvalues}
    \end{align}
\end{lemma}
\begin{proof}
    We regard $\mathbb{R}^d$ as a Hilbert space with the usual inner product.
    Let ${\rm e}_i \in \mathbb{R}^{n+1}$ be the $i$-th elementary vector whose elements are zero except for the $i$-th component, which is $1$.
    For $\alpha = (\alpha_1,\dots,\alpha_d) \in \mathbb{Z}_{\ge 0}^d$, we define ${\rm e}_\alpha ={\rm e}_{\alpha_1} \otimes \dots \otimes{\rm e}_{\alpha_d} \in (\mathbb{R}^{n+1})^{\otimes d}$.
    Let $W \subset (\mathbb{R}^{n+1})^{\otimes d}$ be the subspace generated by $\{{\rm e}_\alpha\}_{|\alpha| \le n}$ and let $P: (\mathbb{R}^{n+1})^{\otimes d} \to W$ be the orthogonal projection.
    Let $C: (\mathbb{R}^{n+1})^{\otimes d} \to (\mathbb{R}^{n+1})^{\otimes d}$ be the linear map $\bigotimes_{i=1}^d \mathbf{C}_n(p_i, r_i)$ and let $D := PC|_W: W \to W$.
    Then, the inequality \eqref{ineq for prod of eigenvalues} follows from the fact that $\mathbf{D}_n(\mu_{\Delta(p, \mathbf{r})})$ coincides with the representation matrix of $D$ with respect to the basis $( {\rm e}_\alpha)_{|\alpha| \le n}$.
\end{proof}

As a result, we have the following theorem.
\begin{theorem}\label{thm: asymptotic formula for the minimal eigenvalues for general mu}
    We use the same notation as in Lemma \ref{lem: lower bound}.
    Let $\mu$ be a finite positive Borel measure absolutely continuous with respect to the Lebesgue measure on $\mathbb{R}^d$.
    Let $\rho$ be the Radon--Nikodym derivative of $\mu$.
    Assume that $\rho$ satisfies $\mathop{\rm ess.inf}_{\Delta(p, \mathbf{r})}(\rho) > 0$.
    Let 
    \begin{align*}
        d_{p,\mathbf{r}} := \# \left\{ i \in \{1,2, \dots, d\}  : |p_i| + p_i^2 -r_i^2  = 0 \right\}.
    \end{align*}
    Then, we have
    \begin{align}
        \frac{1}{\Lambda_n(\mu)}
        \lesssim 
        \frac{1}{\Lambda_n(\mu_{\Delta(p, \mathbf{r})})}
        \lesssim  n^{d/2 - d_{p,\mathbf{r}}/4}\prod_{i=1}^d \sigma_{p_i, r_i}^{2n}.
    \end{align}
\end{theorem}
\begin{proof}
    The first inequality follows from $\int h \, {\rm d}\mu_{\Delta(p, \mathbf{r})} \lesssim \int h \,{\rm d}\mu$ for any non-negative measurable function $h$ on $\mathbb{R}^d$.
    The second inequality follows from Lemma \ref{lem: lower bound} with Theorem \ref{thm: asymptotic formula of smallest eigenvalues}.
\end{proof}

\section{Notation table} \label{sec: notation table}

\begin{longtable}{p{0.25\textwidth}p{0.7\textwidth}}
\caption{List of notation}\label{tab:notation}\\
\toprule
\textbf{Notation} & \textbf{Meaning} \\
\midrule
\endfirsthead

\multicolumn{2}{c}{{\tablename\ \thetable{} -- continued from previous page}}\\
\toprule
\textbf{Notation} & \textbf{Meaning} \\
\midrule
\endhead

\midrule
\multicolumn{2}{r}{{Continued on next page}}\\
\bottomrule
\endfoot

\bottomrule
\endlastfoot

$\mathrm{i}$ & Imaginary unit of $\mathbb{C}$. \\[0.4em]

$\overline{a}$ & Complex conjugate of $a \in \mathbb{C}$. \\[0.4em]

$|a|$ & Absolute value (modulus) of $a \in \mathbb{C}$. \\[0.4em]

$S_{>0}$ & Set of positive elements of $S \subset \mathbb{R}$. \\[0.4em]

$S_{\ge 0}$ & Set of non-negative elements of $S \subset \mathbb{R}$. \\[0.4em]

${\rm e}_i$ & The $i$-th elementary vector in $E$, i.e.\ the vector whose components are $0$ except for the $i$-th component, which is $1$. \\[0.4em]

$z_i$ & The $i$-th component of $z \in E$, for $i=1,\dots,d$. \\[0.4em]

${\rm span}(S)$ & Linear subspace generated by a subset $S$ of a vector space. \\[0.4em]

$\mathbb{K}^{m \times n}$ & Set of $m \times n$ matrices with coefficients in $\mathbb{K} = \mathbb{R}$ or $\mathbb{C}$. \\[0.4em]

$\mathbf{I}_n$ & Identity matrix in $\mathbb{C}^{n \times n}$ (subscript $n$ is omitted if no confusion arises). \\[0.4em]

$\mathbf{O}_n$ & Zero matrix in $\mathbb{C}^{n \times n}$ (subscript $n$ is omitted if no confusion arises). \\[0.4em]

$\mathbf{A}^*$ & Adjoint matrix of a matrix $\mathbf{A}$. \\[0.4em]

$\mathbf{A}^\top$ & Transpose matrix of a matrix $\mathbf{A}$. \\[0.4em]

$\mathbf{A}^\dagger$ & Moore--Penrose pseudo-inverse of a matrix $\mathbf{A}$. \\[0.4em]

$\|\mathbf{A}\|_{\rm Fr}$ &
Frobenius norm of $\mathbf{A} \in \mathbb{C}^{m \times n}$,
\[
\|\mathbf{A} \|_{\rm Fr} := \bigl({\rm tr}(\mathbf{A}^*\mathbf{A})\bigr)^{1/2}.
\] \\[0.4em]

$\|\mathbf{A}\|_{\rm op}$ &
Operator norm of $\mathbf{A} \in \mathbb{C}^{m \times n}$,
\[
\|\mathbf{A} \|_{\rm op} := \sup_{\substack{\mathbf{v} \in \mathbb{C}^{n\times 1} \\ \|\mathbf{v}\|_{\rm Fr} = 1}} \|\mathbf{A}\mathbf{v}\|_{\rm Fr}.
\] \\[0.4em]

$\alpha = (\alpha_1,\dots,\alpha_d)$ & A multi-index in $\mathbb{Z}_{\ge 0}^d$. \\[0.4em]

$|\alpha|$ & Length of a multi-index $\alpha = (\alpha_1,\dots,\alpha_d) \in \mathbb{Z}_{\ge 0}^d$,
\[
|\alpha| := \sum_{i=1}^d \alpha_i.
\]
(When used, no confusion with the absolute value on $\mathbb{C}$ will occur.) \\[0.4em]

$c^\alpha$ & For $c = (c_1,\dots,c_d) \in \mathbb{C}^d$ and $\alpha \in \mathbb{Z}_{\ge 0}^d$,
\[
c^\alpha := \prod_{i=1}^d c_i^{\alpha_i}.
\] \\[0.4em]

$\alpha!$ & For $\alpha \in \mathbb{Z}_{\ge 0}^d$,
\[
\alpha! := \prod_{i=1}^d \alpha_i!.
\] \\[0.4em]

${\rm int}(K)$ & Interior of a subset $K$ of a topological space. \\[0.4em]

$K_1 \Subset K_2$ & $K_1$ is relatively compact in $K_2$, i.e.\ $K_1 \subset {\rm int}(K_2)$. \\[0.4em]

$\mu|_Y$ & Restriction of a measure $\mu$ on a measurable space $(X,\mu)$ to a measurable subset $Y \subset X$:
\[
\mu|_Y(A) := \mu(Y \cap A).
\] \\[0.4em]

$\operatorname*{ess\,inf}_Y(\rho)$ &
Essential infimum of a measurable function $\rho : X \to \mathbb{R}$ on a measurable set $Y \subset X$:
\[
\operatorname*{ess\,inf}_Y(\rho) := \sup\left\{ a \in \mathbb{R} : \mu\bigl(\rho^{-1}((-\infty, a)) \cap Y\bigr) = 0 \right\}.
\] \\[0.4em]

${\rm Prob}[A]$ & Probability of an event $A$. \\[0.4em]

$a_\lambda \lesssim b_\lambda$ &
For two families $\{a_{\lambda}\}_{\lambda \in \Lambda}$ and
$\{b_{\lambda}\}_{\lambda \in \Lambda}$ of real numbers, this means that
there exists $C>0$ such that for all $\lambda \in \Lambda$,
$a_{\lambda} \le C b_{\lambda}$. \\[0.4em]

$V'$ & Dual space of a topological vector space $V$ over $\mathbb{C}$, i.e.\ the space of all continuous linear maps from $V$ to $\mathbb{C}$, equipped with the strong topology. \\[0.4em]

$L'$ & Dual map of a continuous linear map $L : V_1 \to V_2$, i.e.\ $L' : V_2' \to V_1'$. \\[0.4em]

$\langle \cdot,\cdot \rangle_H$ & Inner product of a Hilbert space $H$. \\[0.4em]

$\|\cdot\|_H$ & Norm of a Hilbert space $H$ induced by $\langle \cdot,\cdot \rangle_H$. \\[0.4em]

$h^*$ &
For $h \in H$, the unique element of $H'$ defined by
\[
h^*(g) = \langle g, h\rangle_{H}, \qquad g \in H.
\] \\[0.4em]

$\mathbf{e}_p$ &
For $p \in E$, the function $\mathbf{e}_p \in \mathcal{O}_E(E)$ defined by
\[
\mathbf{e}_p(z) := e^{p^* z}, \qquad z \in E.
\] \\[0.4em]

$\nabla_w h(z)$ &
Directional derivative of $h$ at $z$ in the direction $w \in E$,
\[
\nabla_w h(z) := \lim_{|a| \to 0} \frac{h(z+aw) - h(z)}{a}.
\] \\[0.4em]

$\partial_{z_i}$ & Partial derivative $\partial_{z_i} := \nabla_{{\rm e}_i}$. \\[0.4em]

$\mu_E$ & 
Scaled Lebesgue measure $\pi^{-d}{\rm d}x{\rm d}y$ via the identification $\mathbb{R}^d \times \mathbb{R}^d \cong \mathbb{C}^d;~(x,y) \mapsto x + \mathrm{i}y$. 
\\[0.4em]

$\mathcal{O}_E(U)$ & Space of holomorphic functions on an open subset $U \subset E$, equipped with the topology of uniform convergence on compact subsets. \\[0.4em]

$\mathcal{O}_E(S)$ &
For a locally closed subset $S \subset E$, the space of analytic functions on $S$, defined as the inductive limit of $\{\mathcal{O}_E(U)\}_U$ with restriction maps over open $U \subset E$ containing $S$ as a closed subset. \\[0.4em]

\end{longtable}

\end{document}